\documentclass[10pt]{article}
\usepackage{amssymb,amsmath,latexsym,amscd}
\usepackage{amsthm}

\newtheorem{theorem}[equation]{Theorem}
\newtheorem{corollary}[equation]{Corollary}
\newtheorem{lemma}[equation]{Lemma}
\newtheorem{proposition}[equation]{Proposition}
\theoremstyle{definition}\newtheorem{definition}[equation]{Definition}
\theoremstyle{definition}\newtheorem{remark}[equation]{Remark}
\theoremstyle{remark}\newtheorem{example}[equation]{Example}
\numberwithin{equation}{section}

\newcommand{\ol}{\overline}
\newcommand{\Z}{\mathbb{Z}}
\newcommand{\KK}{\mathbb{K}}
\newcommand{\LL}{\mathbb{L}}

\newcommand{\FF}{\mathbb{F}}
\newcommand{\OO}{\mathbb{O}}

\newcommand{\fg}{{\mathfrak g}}
\newcommand{\fb}{{\mathfrak b}}
\newcommand{\fh}{{\mathfrak h}}

\newcommand{\CC}{{\mathbb C}}

\newcommand{\bmu}{{\mathbf \mu}}
\newcommand{\bB}{{\mathbf B}}
\newcommand{\bT}{{\mathbf T}}
\newcommand{\bH}{{\mathbf H}}
\newcommand{\bG}{{\mathbf G}}
\newcommand{\bC}{{\mathbf C}}

\newcommand{\Aut}{{\rm Aut}}

\newcommand{\Id}{{\rm Id}}

\newcommand{\Ad}{{\mathrm {Ad}}}

\newcommand\limind{\mathop{\oalign{lim\cr\hidewidth$\longrightarrow$\hidewidth\cr}}}

\newcommand\bOut{\text{\bf Out}}

\newcommand{\bGL}{{\mathbf{GL}}}
\newcommand{\bPGL}{{\mathbf{PGL}}}
\newcommand{\bAut}{{\mathbf{Aut}}}

\begin{document}
\title{Classification of Quantum Groups via Galois Cohomology}
\author{Eugene Karolinsky$^{1}$, Arturo Pianzola$^{2,3}$ and Alexander Stolin$^{4}$}
\date{}
\maketitle

$^1$ Department of Pure Mathematics, Kharkiv National University, Kharkiv, Ukraine.

$^2$ Department of Mathematical \& Statistical Science, University of Alberta, Edmonton, Alberta, T6G 2G1, Canada.

$^3$ Centro de Altos Estudios en Ciencias Exactas, Avenida de Mayo 866, (1084), Buenos Aires, Argentina.

$^4$ Department of Mathematical Sciences, Chalmers University of Technology and the University of Gothenburg, 412 96 Gothenburg, Sweden.

\begin{abstract}
The first example of a quantum group was introduced by P.~Kulish and N.~Reshetikhin. In the paper
\cite{KR},
they found a new algebra which was later called $U_q (\mathfrak{sl}(2))$. Their example was developed independently by V.~Drinfeld and M.~Jimbo, which
resulted in the general notion of quantum group.
Later, a complimentary approach to quantum groups was developed by
L.~Faddeev, N.~Reshetikhin, and L.~Takhtajan in
\cite{FRT}.

Recently, the so-called Belavin--Drinfeld cohomology (twisted and non-twisted)
have been introduced in the literature to study and classify certain families of quantum groups
and Lie bialgebras.
Later, the last two authors interpreted non-twisted Belavin--Drinfeld cohomology in terms of
non-abelian Galois cohomology $H^1(\FF, \bH)$ for a suitable algebraic $\FF$-group $\bH$.
Here $\FF$ is an arbitrary field of zero characteristic. The non-twisted case is thus fully understood
in terms of Galois cohomology.

The twisted case has only been studied using Galois cohomology for the so-called (``standard'') Drinfeld--Jimbo structure.

The aim of the present paper is to extend these results to all twisted Belavin--Drinfeld cohomology and thus, to present classification of quantum groups in terms of Galois cohomology and the so-called orders. Low dimensional cases $\mathfrak{sl}(2)$ and $\mathfrak{sl}(3)$ are considered in more details
using a theory of cubic rings developed by B.~N.~Delone and D.~K.~Faddeev in \cite{DF}.

Our results show that there exist yet unknown quantum groups for Lie algebras of the types $A_n, D_{2n+1}, E_6$, not mentioned in \cite{ESS}.
 \noindent \\
{\em Keywords:} Belavin--Drinfeld, Yang--Baxter, Quantum group, Lie bialgebra, Galois cohomology. \\
{\em MSC 2000} Primary 20G10, 17B37, 17B62, and 17B67. Secondary 17B01.
\end{abstract}

\section{Introduction}
The ``linearization problem'' in quantum groups, proposed by  Drinfeld \cite{DUN}, and solved in the seminal work of Etingof and  Kazhdan \cite{EK1} and \cite{EK2}, leads naturally (see \cite{KKPS2} for details) to the study of Lie bialgebra structures where the underlying Lie algebra is a finite dimensional split simple Lie algebra $\fg(\KK)$ over the (algebraic) Laurent series field $\KK = \mathbb{C}((t))$. The classification of the Lie bialgebra structures that such
an algebra $\fg(\KK)$ can carry is closely related to the structure of its Drinfeld double. Indeed, the double of $\fg(\KK)$ is always
a Lie algebra of the form $\fg(\KK) \otimes_\KK \mathbb A$, where $\mathbb A$ is either $\KK \times \KK$, the quadratic field extension $\LL=\mathbb{C}((j))$, where $j = t^{\frac{1}{2}}$, or, finally, the algebra
$\KK[\epsilon]$ of dual numbers of $\KK$. The latter case is related to Frobenius algebras and will not be discussed in the present work.
For the first two cases (see again \cite{KKPS2} for details and further references), the classification is given in terms of what the authors call non-twisted and twisted Belavin--Drinfeld cohomology, and the corresponding Lie bialgebra structures are called of non-twisted and of twisted type respectively.
It was also noticed in \cite{KKPS3} that certain non-twisted Belavin--Drinfeld cocycles are Galois cocycles.

The general connection between Belavin--Drinfeld and Galois cohomology were found in \cite{PS}.
The main ingredient of the appearance of the Galois cohomology in the quantum groups theory is the study of the centralizers $\bC(\bG, r)\subset \bG$. Here $\bG$ is the an algebraic $\KK$-group  corresponding to $\fg(\KK)$ and
$r$ is an $r$-matrix, a solution of the modified classical Yang-Baxter equation classified by Belavin and Drinfeld
in \cite{BD}, which we denote by $r_{\rm BD}$.

The main results of \cite{PS} assert that:
\medskip

(a) Non-twisted Belavin--Drinfeld cohomology $H(\bG, r_{\rm BD})$ introduced in \cite{KKPS2} are nothing but the usual Galois cohomology $H^1(\KK, \bC(\bG, r_{\rm BD}))$.

(b) For the Drinfeld--Jimbo $r$-matrix $r_{\rm DJ}$ (it will be defined later), the twisted Belavin--Drinfeld cohomology  can be interpreted in terms of the ordinary Galois cohomology
$H^1(\KK, \widetilde{\bC}(\bG, r_{\rm DJ}))$, where $\widetilde{\bC}(\bG, r_{\rm DJ})$ is a twisted form of the $\KK$-algebraic group
$\bC(\bG, r_{\rm DJ})$ split by the quadratic extension $\LL$ mentioned above
(however, this result was obtained in the case $\bG$ is a group of the adjoint type).

(c) $H^1(\KK, \bC(\bG, r_{\rm DJ})) = 1$ (by Hilbert 90)
and  $H^1(\KK, \widetilde{\bC}(\bG, r_{\rm DJ})) = 1$
(by a theorem of Steinberg, a result that is also used to establish the correspondence mentioned in (b) above).

(d) In \cite{KKPS2}, \cite{KKPS3}, \cite{SP},  non-twisted and twisted Belavin--Drinfeld cohomology $H(\bG, r_{\rm BD})$ and
${\ol H}(\bG, r_{\rm BD})$ were computed for the following classical groups:
$\mathbf{GL}, \mathbf{SL}, \mathbf{SO}$, and for the simply connected $\mathbf{Sp}$.
\medskip

The main objective of the present paper is to deal with (b) and (c) for arbitrary Belavin--Drinfeld matrices.
This completes the classification of  the Lie bialgebras under consideration. We also discuss the classification problem of the corresponding quantum groups.

\begin{remark}
In defining Belavin--Drinfeld cohomology, the group $\bG$ need not be adjoint. In the non-twisted case, the base field $\KK$ could be taken to be arbitrary (of characteristic $0$), and some interesting results can de derived in this generality. In the twisted case, the quadratic extension is crucial. So is the fact that $\KK$ is of cohomological dimension $1$ and that its Galois group is pro-cyclic. These facts, together with the connection with quantum groups for the case of $\KK = \mathbb{C}((t))$, explains why we will restrict our attention to this particular base field.
\end{remark}

The rest of the paper is organized as follows. After establishing some notation in Section \ref{sec_Notation} and reminding the readers the Belavin--Drinfeld classification in Section \ref{sec_BD}, we pass to the main part of the paper. In Section \ref{sec_twisted_1} we define and discuss some basic properties of twisted Belavin--Drinfeld cohomology. In Section \ref{sec_twisted_2} we establish a connection between twisted Belavin--Drinfeld cohomology and Galois cohomology of a twisted quasitorus. In Section \ref{sec_classif_alg} we apply the above results to classify the corresponding Lie bialgebra structures on $\fg(\KK)$ up to gauge equivalence. In Section \ref{sec_classif_gr} we classify the corresponding quantum groups in terms of certain double cosets in the group $\bG(\KK)$. In Appendix \ref{appendix_A}, written by by Juliusz Brzezinski and A.~S., the theory of orders is applied to describe the double cosets mentioned above in the case $\bG=\mathbf{GL}(n)$. Finally, in Appendix \ref{appendix_B}, written by E.~K. and Aleksandra Pirogova, Belavin--Drinfeld cohomology for exceptional simple Lie algebras are discussed.

\section{Notation}\label{sec_Notation}

Throughout this paper $\mathbb{K}$ will denote $\mathbb{C}((t))$ and  $\mathbb{L}$ its quadratic extension $\mathbb{C}((j))$, where $j = t^{\frac{1}{2}}$.
We fix an algebraic closure of $\KK$, which will be denoted by $\overline{\KK}$. The (absolute)
Galois group $\rm{Gal}(\KK)$ of the extension $\overline{\KK}/\KK$ will be denoted
by $\mathcal{G}$. For future reference we recall the explicit description of  $\mathcal {G}$.

Fix a compatible set of primitive
$m^{\rm th}$ roots of unity $\xi_m$, namely such that  $\xi _{me} ^e = \xi_m$ for all integer
$e > 0$.\footnote{For example, $\xi_m = e^{\frac{i2\pi}{m}}$.}
Fix also, with  obvious meaning, a compatible set $t^\frac{1}{m}$
of $m^{\rm th}$ roots of $t$ in $\overline{\KK}$.

Let $\KK_{m} =  \mathbb{C}((t^\frac{1}{m}))$. Then we can  identify $\text{\rm Gal}(\KK_m/\KK)$
with $\Z/m\Z$, where for each $e \in \Z$ the corresponding element $\ol{e} \in \Z/m\Z$  acts on $\KK_m$
via $ ^{\ol {e}} t^{\frac{1}{m}}_i = \xi^{e}_{m}
t^{\frac{1}{m}}_i$.
\smallskip

We have $\overline{\KK}  = {\limind} \,\, \KK_m$. The absolute Galois group $\mathcal{G}$ of $\KK$ is the
profinite completion $\widehat{\Z}$ understood as the inverse limit of the Galois groups $\text{\rm Gal}(\KK_m/\KK)$
as described above. If $\gamma_1$ denotes the standard profinite
generator of $\widehat{\Z}$, then the action of $\gamma_1$ on $\overline{\KK}$ is given by
$$^{\gamma_1}t^\frac{1}{m} = \xi_mt^\frac{1}{m}.$$
Note for future reference that $\gamma_2 := 2\gamma_1$ is the canonical profinite generator of $ \mathcal{G}_\LL = \text{\rm Gal}(\LL)$.

If $V$ is a $\KK$-space (resp.\ Lie algebra), we will denote the $\overline{\KK}$-space (resp.\ Lie algebra) $V \otimes_\KK \overline{\KK}$ by $\overline{V}$.

If $\text{\bf K}$ is a (smooth) linear algebraic group over $\KK$, then  the corresponding (non-abelian) \'{e}tale Galois cohomology will be denoted by $H^1(\KK, \text{\bf K})$ (see \cite{Se} for details). We recall that $H^1(\KK, \text{\bf K})$ coincides with the usual non-abelian continuous cohomology of the profinite group $\mathcal{G}$ acting (naturally)  on $\text{\bf K}(\overline{\KK})$.

Let  $\fg$ be a split finite dimensional simple Lie algebra over $\mathbb{C}$, $\fg(\KK)=\fg\otimes_{\mathbb C}\KK$. In what follows the adjoint group of $\fg(\KK)$ (viewed as an algebraic group over $\KK$) will be denoted by $\bG_{\rm ad}$.

We fix once and for all a Killing couple $(\bB_{\rm ad}, \bH_{\rm ad})$ of $\bG_{\rm ad}$,
whose corresponding Borel and split Cartan subalgebras will be denoted by $\fb$ and $\fh$ respectively.
Our fixed Killing couple leads, both at the level of $\bG_{\rm ad}$ and $\fg(\KK)$, to a root system $\Delta$ with a fixed
set of positive roots $\Delta_+$ and the base
$\Gamma = \{ \alpha_1, \ldots, \alpha_n \}$.\footnote{The elements of $\Delta$ are to be thought as characters of $\bH_{\rm ad}$ or elements of $\fh^*$ depending on whether we are working at the group or Lie algebra level. This will always be clear from the context.}

The Lie bialgebra structures that we will be dealing with are defined by $r$-matrices, which are  elements of $\fg(\KK) \otimes_\KK \fg(\KK)$ satisfying  ${\rm CYB}(r) = 0$ where ${\rm CYB}$ is the classical Yang--Baxter
operator (see \S3 below and \cite{ES} for definitions).

The  action of $\bG_{\rm ad}$ on $\fg(\KK) \otimes_\KK \fg(\KK)$ induced by the adjoint action of $\bG_{\rm ad}$ on $\fg(\KK)$ will be denoted by $\Ad_X$.
Along similar lines, if $\sigma \in \cal{G}$, then  we will write $\sigma(r)$ instead of $(\sigma \otimes \sigma)(r)$.

Fix $r \in \fg(\KK) \otimes_\KK \fg(\KK)$. The centralizer of $r$ in $\bG_{\rm ad}$ (under the adjoint action) will be denoted by $\bC(\bG_{\rm ad}, r)$. It is an algebraic $\KK$-group and a closed subgroup of
$\bG_{\rm ad}$. Its functor of points is as follows. Let $R$ be a commutative ring extension of $\KK$. View $r$ as an element of $(\fg(\KK) \otimes_\KK\fg(\KK)) (R) = (\fg(\KK) \otimes_\KK \fg(\KK))\otimes_\KK  R \simeq (\fg(\KK) \otimes_\KK R) \otimes_R (\fg(\KK) \otimes_\KK R)$ in a natural way. Then
$$\bC(\bG_{\rm ad}, r)(R) = \{ X \in \bG_{\rm ad}(R) : \Ad_X(r) = r \} .$$

\section{Belavin--Drinfeld classification}\label{sec_BD}

Let $\FF$ be an arbitrary field extension of $\mathbb C$. For the time being we replace $\KK$ by $\FF$.

Consider a Lie bialgebra structure $\delta$  on $\fg(\FF)$. By Whitehead's Lemma the cocycle $\delta : \fg(\FF) \to \fg(\FF) \otimes_\FF \fg(\FF)$  is a coboundary.
Thus, $ \delta = \delta_r$ for some element $r \in \fg(\FF) \otimes_\FF \fg(\FF)$, namely
$$
\delta (a)=[r, a\otimes 1+1\otimes a]
$$
for all $a \in \fg(\FF)$. It is well known when an element $r \in \fg(\FF) \otimes_\FF \fg(\FF)$ determines a Lie bialgebra structure of $\fg(\FF)$. See \cite{ES} for details.

We assume until further notice that $\FF$ is algebraically closed. Then we have the Belavin--Drinfeld classification \cite{BD},
which is useful to recall now. Following \cite{BD}, we define an equivalence relation between two $r$-matrices  $r, r' \in\fg(\FF) \otimes_{\FF} \fg(\FF)$ by declaring that $r$ is equivalent to $r'$ if there exist an element $X \in \bG_{\rm ad}(\FF )$ and a scalar $b \in \FF^\times$ such that
\begin{equation}\label{equivalent}
r' = b \, \Ad_X(r).
\end{equation}
Furthermore, if  $b = 1$,  these two $r$-matrices are called gauge equivalent.

Belavin and Drinfeld  provide us with a list of elements $r_{\rm BD} \in \fg(\FF) \otimes_\FF \fg(\FF)$ (called Belavin--Drinfeld r-matrices) with the following properties:
\begin{enumerate}
\item Each $r_{\rm BD}$ is an $r$-matrix (i.e.\ a solution of the classical Yang--Baxter equation) satisfying  $r + r^{21} = \Omega$ (where  $\Omega$ is the Casimir operator
  of $\fg(\FF) \otimes_\FF \fg(\FF)$).

\item Any non-skewsymetric $r$-matrix for $\fg(\FF)$ is equivalent to a unique $r_{\rm BD}$.
\end{enumerate}

For the readers' convenience we recall the structure of the Belavin--Drinfeld $r$-matrices.
With respect to our fixed  $(\fb, \fh)$, any $r_{\rm BD}$ depends on a discrete and a continuous parameter.
The discrete parameter is an admissible triple $(\Gamma_{1},\Gamma_{2},\tau)$, which is
an isometry $\tau:\Gamma_{1}\to \Gamma_{2}$. Here $\Gamma_{1},\Gamma_{2}\subset\Gamma$ and
for any $\alpha\in\Gamma_{1}$ there exists $k\in \mathbb{N}$ satisfying
$\tau^{k}(\alpha)\notin \Gamma_{1}$. The continuous parameter is a tensor $r_{0}\in \mathfrak{h} \otimes_\FF\mathfrak{h}$ satisfying $r_{0}+r_{0}^{21}=\Omega_{0}$
and $(\tau(\alpha)\otimes 1+1 \otimes \alpha)(r_{0})=0$ for any $\alpha\in \Gamma_{1}$.
Here $\Omega_{0}$ denotes the Cartan part of the quadratic Casimir element $\Omega$.
Then
\begin{equation}\label{rBD}
r_{\rm BD} =r_{0}+\sum_{\alpha>0}e_{\alpha}\otimes e_{-\alpha}+\sum_{\alpha\in (\mathrm{Span}\,\Gamma_{1})^{+} }\sum_{k\in \mathbb{N}} e_{\alpha}\wedge e_{-\tau^{k}(\alpha)}.
\end{equation}
where $e_\alpha$ and $e_{-\alpha}$ are parts of a fixed Chevalley system of $(\fg, \fh)$ in the sense of \cite[Ch.\ VIII, \S2 and \S12]{Bbk}. We will sometimes write $r_{\rm BD} = r_{0}+r_{\rm BD}'$.

We return to the case of our field  $\KK=\CC ((t))$.
Let $\delta$ be a
Lie bialgebra structure on $\fg(\KK)$. Clearly, it is of the form
$\delta (a)= \delta_r (a)=[r, a\otimes 1+1\otimes a],$ where $a\in \fg(\KK)$ and
$r \in \overline{\fg}\otimes_{\overline{\KK }}\overline{\fg}$ is an $r$-matrix.
We will assume that $(\fg(\KK), \delta)$ is
not triangular, i.e.\ $r$ is not skew-symmetric.

By the Belavin--Drinfeld classification there exists a unique  $r_{\rm BD}$ such that
\begin{equation}\label{crational}
{r} = b \, \mathrm{Ad}_{X}(r_{\rm BD})
\end{equation}
for some $X\in \bG_{\rm ad}(\overline{\KK})$  and $b \in \overline{\KK}^\times$. Since ${r} + {r}^{21} = b\,\Omega$,
we can apply \cite[Theorem 2.7]{KKPS3}  to conclude that $b^2 \in \KK$.

This leads to two cases, depending on whether $b$ is in $\KK$ or not. The first case is treated with the non-twisted Belavin--Drinfeld
cohomology, and it is dealt in full generality by means of the Galois cohomology $H^1(\KK, \bC(\bG, r_{\rm BD}))$ in \cite{PS}.

Our interest is in the second case with $b = j$.   The corresponding twisted Belavin--Drinfeld cohomology and their relation to quantum groups and Galois cohomology are the contents of the next two sections.

\begin{definition}
The discrete parameter $(\Gamma_1, \Gamma_2, \tau)$ of the unique Belavin--Drinfeld matrix $r_{\rm BD}$ in (\ref{crational}) will be called {\it the discrete parameter of} $r$.
\end{definition}

\begin{lemma}\label{trivialr}
Let $r \in \overline{\fg}\otimes_{\overline{\KK }}\overline{\fg}$ be an $r$-matrix. Then

(i) $r^{21}$ is an $r$-matrix.

(ii) $\gamma({r})$ is an $r$-matrix for all $\gamma \in \mathcal{G}$.

(iii) Let $r_{\rm BD}$ be a Belavin--Drinfeld matrix as in (\ref{rBD}). Then
$\gamma(r_{\rm BD})$ is also a Belavin--Drinfeld matrix for all $\gamma \in \mathcal{G}$. Furthermore, these two $r$-matrices differ only on their continuous parameter. In particular, they have the same discrete parameter.
\end{lemma}

\begin{proof}
The first statement is well known, the second and third are obvious.
\end{proof}

\section{Twisted Belavin--Drinfeld cohomology}\label{sec_twisted_1}

In the remainder of our paper we will assume that in (\ref{crational}) we have $b=j=t^{1/2} \in \LL^\times$. Thus,
 \begin{equation}\label{rbar}
 \overline{r} = j \, \mathrm{Ad}_{X}(r_{\rm BD})
 \end{equation}
and
 \begin{equation}\label{rrel}
 \overline{r} + \overline{r}^{21} = j\,\Omega.
 \end{equation}

Recall the following  result proved in \cite[Theorem 3]{KKPS2}:

\begin{theorem}\label{CartanF}
Assume that $\overline{r} = b\, \mathrm{Ad}_{X}(r_{\rm BD})$, $b\in\KK$, induces a Lie bialgebra structure on $\fg(\KK)$. Then both $r_{\rm BD}$  and $\overline{r}$  are rational, i.e.\ they  belong to $\fg(\KK) \otimes_\KK \fg(\KK)$.
Furthermore, $X^{-1}\gamma(X)\in \bC(\bG_{\rm ad}, r_\text{\rm BD})(\overline{\KK})$ for all $\gamma \in \mathcal{G}$. \qed
\end{theorem}

This allows us to establish the following

\begin{proposition}\label{aux}
Assume that $j\mathrm{Ad}_{X}(r_{\rm BD})$ induces a Lie bialgebra structure on $\fg(\KK)$. Then
\begin{enumerate}
\item (i) $\gamma_2 (r_\text{\rm BD}) =r_\text{\rm BD}$,

(ii) $\mathrm{Ad}_{\gamma_2 (X)}(r_{\rm BD})=\mathrm{Ad}_{X}(r_{\rm BD})$, \,\, {\text \rm and}

(iii) $X^{-1}\gamma_2(X)\in \bC(\bG_{\rm ad}, r_\text{\rm BD})(\overline{\KK})$.

\item $\gamma_1 (\mathrm{Ad}_{X}(r_{\rm BD})) =\mathrm{Ad}_{\gamma_1 (X)}(\gamma_1 (r_{\rm BD} ))=(\mathrm{Ad}_{X}(r_{\rm BD}))^{21}$.
  In particular, $r$-matrices $\gamma_1 (r_{\rm BD})$ and $r_{\rm BD}^{21}$
  have the same discrete parameter.
\end{enumerate}
\end{proposition}

We remind the reader that $\gamma_1 $ is a fixed progenerator of $\cal{G}=\mathrm{Gal} (\KK )$ and $\gamma_2 =2\gamma_1$
   is a progenerator of $\mathrm{Gal} (\LL )$.

\begin{proof}
 (1) This follows from  Theorem \ref{CartanF} using $\LL$ instead of $\KK$ as the base field.

 (2) The second statement follows from the following lemma, which will also be  used later.
\end{proof}

\begin{lemma}\label{j}
    Assume that $r$ satisfies the CYBE and $r+r^{21} =\Omega$.
    Then the following two conditions are equivalent:
    \begin{itemize}
     \item (a) $jr$ induces a Lie bialgebra structure on $\fg(\KK)$.
     \item (b) $\gamma_1 (r)=r^{21}$ and $\gamma_2 (r)=r$.
    \end{itemize}
\end{lemma}

\begin{proof}
 (a) $\Rightarrow $ (b).  Indeed, for any $a\in \fg(\KK)$
 $$
 \gamma_2 ([jr, a\otimes 1+1\otimes a])=[jr, a\otimes 1+1\otimes a]=[j\gamma_2 (r), a\otimes 1+1\otimes a].
 $$
 Therefore, $r=\gamma_2 (r)+p\Omega$ and $r^{21}=\gamma_2 (r^{21} )+p\Omega$. Since $\gamma_2 (r+r^{21} )=\Omega $, we see that $p=0$.

 Now, we will prove that $\gamma_1 (r)=r^{21}$. Since $\gamma_1 (j)=-j$, we have
 $$
 \gamma_1 ([jr, a\otimes 1+1\otimes a])=[-j\gamma_1 (r), a\otimes 1+1\otimes a]=[jr, a\otimes 1+1\otimes a].
 $$
 The last equality implies immediately  that $r+\gamma_1 (r) =q \Omega$ with $q\in \overline{\KK}$.
 Applying $\gamma_1$ again to the latter equality, and taking into account that by the first part of the proof we have $2\gamma_1 (r)=r$,
 we see that $q\in {\KK}$.

 Since $r+r^{21} =\Omega$ and  $\gamma_1 (r+r^{21} )=\Omega$, we deduce that $q=1$,
 i.e., that $\gamma_1 (r)=r^{21}$.

 (b) $\Rightarrow $ (a).
To prove that $jr$ induces a Lie bialgebra structure on $\fg(\KK)$,
we have to verify that $\gamma_i ([jr, a\otimes 1+1\otimes a])=[jr, a\otimes 1+1\otimes a]$ for $i=1,2$
and any $a\in \fg(\KK)$.

If $i=2$, then it is clear.
It remains to prove the above statement for $i=1$.
In this case we have:
$$
\gamma_1 ([jr, a\otimes 1+1\otimes a])=[-jr^{21}, a\otimes 1+1\otimes a]=
[jr, a\otimes 1+1\otimes a],
$$
since $jr +jr^{21}=j\Omega$ and $[\Omega, a\otimes 1+1\otimes a]=0$.
\end{proof}

Let $c$ be the Chevalley involution of $(\fg(\KK), \fb, \fh)$. By definition, this is
the unique automorphism of $\fg(\KK)$ that  maps  $e_{\alpha}$ to $e_{-\alpha}$ for all
simple roots $\alpha$. Of course $c^2 = \Id$ and $c$ acts on $\fh$ as $-\Id$.
The following result shows that the last condition of the previous proposition  imposes sharp necessary conditions
for the existence of non-trivial discrete parameters on $r_{\rm BD}$.

\begin{proposition}\label{21discrete}
  Assume that $c\in \bG_{\rm ad}(\KK)$.
  Then the equation $$\gamma_1(\mathrm{Ad}_X(r_{\rm BD}))=(\mathrm{Ad}_{X}(r_{\rm BD}))^{21}$$ has no solutions
  unless the admissible triple  for $r_{\rm BD}$ satisfies  $\Gamma_1 =\Gamma_2 = \emptyset$.
\end{proposition}

\begin{proof}
Let $(\Gamma_1, \Gamma_2, \tau)$ be the discrete parameter of $r_{\rm BD}$. First of all, let us
notice that $\mathrm{Ad}_{\gamma_1 (X)}(\gamma_1 (r_{\rm BD} ))$ has the same discrete parameter as  $r_{\rm BD}$.
Indeed, this is true for $\mathrm{Ad}_{\gamma_1 (X)}(\gamma_1 (r_{\rm BD} ))$ and $\gamma_1(r_{\rm BD})$ by definition,
and for $\gamma_1(r_{\rm BD})$ and $r_{\rm BD}$ by Lemma \ref{trivialr}. Our assumption then implies that
the discrete parameter of the $r$-matrix $r_{\rm BD}^{21}$ is $(\Gamma_1, \Gamma_2, \tau)$.

We claim, however, that the discrete parameter of $r_{\rm BD}^{21}$ is $(\Gamma_2, \Gamma_1, \tau^{-1})$.
This clearly forces $\Gamma_1 = \Gamma_2 = \emptyset$.

Since $c \in \bG_{\rm ad}(\KK)$, the discrete parameter of $r_{\rm BD}^{21}$ coincides with the discrete parameter
of $\mathrm{Ad}_c (r_{\rm BD} )^{21}$. Since $\Ad_c (h_1\otimes h_2)=h_1\otimes h_2$ for any $h_{1}, h_{2}\in\fh$
(because $c$ acts on $\fh$ by  $-\Id$), we have
$$\mathrm{Ad}_c (r_{\rm BD} )^{21} = r_{0}^{21}+\sum_{\alpha>0}e_{\alpha}\otimes e_{-\alpha}+\sum_{\alpha\in (\mathrm{Span}\,\Gamma_{1})^{+} }\sum_{k\in \mathbb{N}} e_{\tau^k(\alpha)}\wedge e_{-\alpha}.$$
Since $\tau^k(\alpha)$ belongs to the span of $\Gamma_2$,
the discrete parameter of this last $r$-matrix is  $(\Gamma_2 , \Gamma_1, \tau^{-1})$ as claimed.
\end{proof}



\begin{remark}
The Chevalley involution is inner, i.e.\ $c \in \bG_{\rm ad}(\KK)$, if and only if  $\fg$ is of type $A_1, B_n , C_n , D_{2n} , G_2, F_4, E_7, E_8$.
\end{remark}

\begin{remark}
The last proposition says {\it nothing} about the existence of an element $X \in \bG_{\rm ad}(\overline{\KK})$ for which $r = j\Ad_X(r_{\rm BD} )$ generates a Lie bialgebra structure on
$\fg(\KK)$. What it does say is that, if such  an $X$ exists and $c$ is inner, then $r_{\rm BD}$ must necessarily have a trivial discrete parameter.
The rest of the paper deals with the  existence and classification of such elements.
\end{remark}

Let ${\rm Out}(\fg)$ be the finite group of automorphisms of the Coxeter--Dynkin diagram of our simple Lie algebra $\fg(\KK)$. If $\bOut(\fg)$
is the corresponding constant $\KK$-group, we know \cite{SGA3} that there exists a split exact sequence of algebraic $\KK$-groups
\begin{equation}\label{split}
1 \to \bG_{\rm ad} \to \bAut(\fg) \to \bOut(\fg) \to 1.
\end{equation}
We fix a section $\bOut(\fg) \to \bAut(\fg)$ that stabilizes $(\bB, \bH)$.
This gives a  copy of ${\rm Out} (\fg) = \bOut(\fg)(\KK)$ inside ${\rm Aut} (\fg) = \bAut(\fg)(\KK)$
that permutes the fundamental root spaces $\fg(\KK)^{\alpha_i}$, and which stabilizes both  our chosen
Borel and Cartan subalgebras. Of course, ${\rm Aut} (\fg)$ is the semi-direct product of $\bG_{\rm ad}(\KK)$ and ${\rm Out}(\fg)$.

As explained in \cite[Lemma 5.9]{PS}, if the Chevalley involution $c$ is not inner, then there exists an element $d\in {\rm Out}(\fg)$ of order $2$ such that $cd = S$
is an inner automorphism of $\fg(\KK)$. The elements $c$ and $d$ commute, hence, $S$ has order $2$. Of course, if $c$ is inner, then $d = \Id$ and $c = S$.


\begin{proposition}\label{X}
Let $S=cd$ be as above. Let $r_{\rm BD}$ be a Belavin--Drinfeld matrix with discrete parameter
$(\Gamma_1, \Gamma_2, \tau)$ and continuous parameter $r_0$.
Assume that $j\mathrm{Ad}_{X}(r_{\rm BD})$ induces a Lie bialgebra structure on $\fg(\KK)$ for some $X \in \bG_{\rm ad}(\overline{\KK})$. Then the following four conditions are satisfied:
\begin{enumerate}
  \item $\Gamma_2 =d(\Gamma_1)$,
  \item $\tau=d\tau^{-1}d^{-1}$,
  \item $\gamma_2 (r_{\rm BD})=r_{\rm BD} .\ {\rm In\ particular,}$
    $\gamma_2 (r_0)=r_0$,
  \item  $\gamma_1 (r_0) = {\rm Ad}_S (r_0)^{21}$.
\end{enumerate}
\end{proposition}

\begin{proof}
By Proposition \ref{aux} (2) we have $\gamma_1(\mathrm{Ad}_X(r_{\rm {BD}}))=(\mathrm{Ad}_{X}(r_{\rm {BD}}))^{21}$. Then it follows from Lemma \ref{trivialr} that $\gamma_1(r_{\rm BD})$ and $\Ad_S(r_{\rm BD}^{21})$ have the same discrete parameter.
We know that the discrete parameter of $\gamma_1(r_{\rm BD})$ is $(\Gamma_1, \Gamma_2, \tau)$.
The same reasoning that we used in the proof of Proposition \ref{21discrete} shows that the discrete parameter of $\Ad_S(r_{\rm BD}^{21})$ is $(d(\Gamma_2) , d(\Gamma_1) , d \tau^{-1} d^{-1})$.
This proves the first two statements.\footnote{If $c$ is inner, these two statements are clear. Indeed, $d = \Id$ and by Proposition \ref{21discrete} $\Gamma_1 = \Gamma_2 = \emptyset$. By convention, $\tau =\Id$.}

(3) is a direct consequence of Proposition \ref{aux} (1.i).

Let us prove (4) now. First, we observe that ${\rm Ad}_S (r_{\rm BD})^{21}$ and  $\gamma_1(r_{\rm BD} )$ belong to the Belavin--Drinfeld list. Since
$j\mathrm{Ad}_{X}(r_{\rm BD})$ induces a Lie bialgebra structure on $\fg(\KK)$,
by Lemma \ref{j} we have
\begin{equation}
\mathrm{Ad}_{X^{-1}\gamma_1 (X)}(\gamma_1(r_{\rm BD} ))=r_{\rm BD}^{21}.
\end{equation}
Hence, we get the following equality:
\begin{equation}\label{sasha}
\mathrm{Ad}_{SX^{-1}\gamma_1 (X)}(\gamma_1(r_{\rm BD} ))=\mathrm{Ad}_{S}(r_{\rm BD})^{21}.
\end{equation}
Thus, two r-matrices from the Belavin--Drinfeld list, $\gamma_1(r_{\rm BD} )$ and ${\rm Ad}_S (r_{\rm BD})^{21}$,
are gauge equivalent and, therefore, equal (by the Belavin--Drinfeld classification).
In particular, their continuous parameters are equal, which proves that $\gamma_1 (r_0) = {\rm Ad}_S (r_0)^{21}$.
\end{proof}

\begin{corollary}\label{C}
  Assume that $j{\rm Ad}_{X} (r_{\rm BD})$ induces a Lie bialgebra structure on $\fg(\KK)$.
  Then

  (1) $SX^{-1}\gamma_1 (X)\in \bC (r_{\rm BD}, \bG_{\rm ad} )(\overline{\KK})$.

  (2) $\bC (r_{\rm BD}, \bG_{\rm ad} )(\overline{\KK})$ is stable under the action of $\Ad_S$.

\end{corollary}


\begin{proof}
The first statement follows from the equality $\gamma_1(r_{\rm BD} )={\rm Ad}_S (r_{\rm BD})^{21}$ and (\ref{sasha}).

The second statement is a consequence of the following facts:
\begin{itemize}
    \item $\bC (r_{\rm BD}, \bG_{\rm ad} )=\bC (\gamma_1 (r_{\rm BD}), \bG_{\rm ad})$,
    \item $\bC (r_{\rm BD}, \bG_{\rm ad} )=\bC (r_{\rm BD}^{21}, \bG_{\rm ad})$,
    \item $\bC (\gamma_1 (r_{\rm BD}), \bG_{\rm ad}) = \bC ({\rm Ad}_S (r_{\rm BD})^{21}, \bG_{\rm ad})  = {\rm Ad}_S (\bC (r_{\rm BD}^{21}, \bG_{\rm ad}))$.
\end{itemize}
\end{proof}

\begin{definition}
  $X\in \bG_{\rm ad}(\overline{\KK})$ is called a \emph{twisted Belavin--Drinfeld cocycle} if there exist $D_1, D_2 \in  \bC (r_{\rm BD}, \bG_{\rm ad})(\overline{\KK})$ such that
  $\gamma_2 (X)=XD_2$
  and $\gamma_1 (X)=XSD_1$.
\end{definition}

The set of all twisted Belavin--Drinfeld cocycles is denoted by
$\overline{Z}(\bG_{\rm ad}, r_{\rm BD})$.

\begin{remark}
Assume that $r_{\rm BD}$ is rational, that is $r_{\rm BD} \in \fg(\KK) \otimes_\KK \fg(\KK)$. Then the above definition of a twisted Belavin--Drinfeld cocycle coincides with the one given in \cite[Definition 5.4]{PS}. Indeed, for $\gamma_2 \in \mathrm{Gal} (\LL ),$ we have $X^{-1}\gamma_2(X) = D_2 \in \bC (r_{\rm BD}, \bG_{\rm ad})(\overline{\KK})$. Finally,
by Proposition \ref{aux},  $\Ad_X(r_{\rm BD}^{21}) = \big(\Ad_X(r_{\rm BD})\big)^{21} = \Ad_{\gamma_1(X)} (\gamma_1(r_{\rm BD})) = \Ad_{\gamma_1(X)}(r_{\rm BD})$.
This yields
$$\Ad_{X^{-1}\gamma_1(X)}(r_{\rm BD}) = r_{\rm BD}^{21}.$$
\end{remark}

Now we are ready to prove that if $r_{\rm BD}$ satisfies the conclusions of Proposition \ref{X}, then
$\overline{Z}(\bG_{\rm ad} , r_{\rm BD}) $ is non-empty. The crucial ingredient of the proof is the existence of the
element $J\in \bG_{\rm ad} (\LL)$ such that $\gamma_1 (J)=JS$, see \cite[Proposition 5.11]{PS}.

\begin{proposition}
Let $r_{\rm BD}$ satisfies the conclusions of Proposition \ref{X}. Then
$j{\rm Ad}_J (r_{\rm BD})$ induces a Lie bialgebra structure on $\fg(\KK)$.
\end{proposition}

\begin{proof}
Lemma \ref{j} implies that we need only
to verify that
$\gamma_2 ({\rm Ad}_J (r_{\rm BD}))={\rm Ad}_J (r_{\rm BD})$
and $\gamma_1({\rm Ad}_J (r_{\rm BD}))= \Ad_J(r_{\rm BD})^{21}$.

The first equality is clear since $J \in \bG_{\rm ad} (\LL)$ and
$\gamma_2 (r_{\rm BD})=r_{\rm BD}$ (because $r_{\rm BD}$ satisfies Proposition \ref{X} (3)).



For the second one we have
$$
  \gamma_1 ({\rm Ad}_J (r_{\rm BD}))={\rm Ad}_J ({\rm Ad}_S (r_{\rm BD}' +\gamma_1 (r_0)))= {\rm Ad}_J (r_{\rm BD}')^{21} + {\rm Ad}_J ({\rm Ad}_S (\gamma_1 (r_0))),
$$
where $r_{\rm BD}'=r_{\rm BD}-r_0$. Here we have used the following facts:
\begin{itemize}
\item $\gamma_1 ({\rm Ad}_J) ={\rm Ad}_{JS}$,
\item $\gamma_1 (r_{\rm BD}')=r_{\rm BD}'$,
\item ${\rm Ad}_S (r_{\rm BD}')=(r_{\rm BD}')^{21}$ (because if $\alpha\in\rm{Span}(\Gamma_1)$, $\beta\in\rm{Span}(\Gamma_2)$ and $\beta=\tau^k (\alpha)$, then
${\rm Ad}_S (e_\alpha\otimes e_{-\beta})= e_{-d(\alpha)}\otimes e_{d(\beta)}$,
$d(\alpha)\in\rm{Span}(\Gamma_2)$, $d(\beta)\in\rm{Span} (\Gamma_1)$, and
$\tau^{-k}(d(\beta))=d(\alpha)$).
\end{itemize}
Since $S^2 = \Id$, then by Proposition \ref{X} (4) we conclude that ${\rm Ad}_S (\gamma_1 (r_0))=(r_0)^{21}$. Hence, we get $\gamma_1({\rm Ad}_J (r_{\rm BD}))= \Ad_J(r_{\rm BD})^{21}$.
\end{proof}

\begin{corollary}
The set $\overline{Z}(\bG_{\rm ad}, r_{\rm BD})$ is non-empty if and only if $r_{\rm BD}$ satisfies the conclusions of Proposition \ref{X}.\qed
\end{corollary}

\begin{remark}
It is not so easy to describe explicitly all discrete parameters in the case  $\fg =A_n$  that satisfy the conclusions
of Proposition \ref{X}. On the other hand, all possible discrete parameters for $D_{2n+1}$
were found in \cite{KKPS3}. There, it was also noticed  that if the discrete parameter satisfies the conclusions
of Proposition \ref{X}, then the set of the corresponding continuous parameters is non-empty.
\end{remark}

\begin{definition}
Two twisted cocycles $X_1$ and $X_{2}$ in $\overline{Z}(\bG_{\rm ad},r_{\rm BD})$ are called \emph{equivalent} if
there exist $Q\in \bG_{\rm ad} (\KK)$ and $C\in \bC(\bG_{\rm ad}, r_{\rm BD})(\overline{\KK})$ such that $X_{1}=QX_{2}C$.
\end{definition}

\begin{definition}
The \emph{twisted Belavin--Drinfeld cohomology} related to $\bG_{\rm ad}$ and $r_{\rm BD}$ is the set of  equivalence classes
of the twisted cocycles. We will denote it by $\overline{H}(\bG_{\rm ad} , r_{\rm BD})$.
\end{definition}

For a motivation of these two definitions see \cite{KKPS1, KKPS2}. The twisted Belavin--Drinfeld cohomology
provides a classification of quantum groups modulo the action of the gauge group
$\bG_{\rm ad} (\KK)$.

\section{From twisted Belavin--Drinfeld cocycles to $H^1$ of a twisted $\KK$-group}\label{sec_twisted_2}
Throughout this section $r_{\rm BD}$ satisfies the conclusions of Proposition \ref{X}.
One of the most important $r$-matrices is the so-called \emph{Drinfeld--Jimbo} one given by

\begin{definition}\label{rDJdefinition}
$r_\text{\rm DJ} = \sum_{\alpha>0}e_{\alpha}\otimes e_{-\alpha} + \frac{1}{2}\, \Omega_0$.
\end{definition}

Here $\Omega_0$, as it has already been mentioned, stands for the $\fh \otimes_\KK \fh$-component of the
Casimir operator $\Omega$ of $\fg(\KK)$ written with respect to our choice of $(\fb,\fh)$.

 Recall that
$\bC(\bG_{\rm ad}, r_{\rm BD})(\overline{\KK})$ is always a closed subgroup of $\bH(\overline{\KK})$ and that
$\bC(\bG_{\rm ad}, r_{\rm DJ})(\overline{\KK}) = \bH(\overline{\KK})$.
\medskip


The following theorem of \cite{PS} plays a crucial role in this part of the paper.

\begin{theorem}\label{maintwisted}
The set $\overline{H} (\bG_{\rm ad}, r_{\rm DJ})$ consists of one element.\qed
\end{theorem}

More precisely, our element $J$ is an element of $\overline{Z} (\bG_{\rm ad}, r_{\rm DJ})$ and any other twisted cocycle is equivalent to $J$. The crucial importance of this result is the following

\begin{corollary}
  Assume that $X\in \overline{Z}(\bG_{\rm ad},r_{\rm BD})$.
  Then $X=QJD$, where $Q\in \bG_{\rm ad} (\KK)$ and $D\in \bH(\overline{\KK})$.
\end{corollary}

\begin{proof}
It was proved in \cite{KKPS2} that
$$\bC(\bG_{\rm ad}, r_{\rm BD})(\overline{\KK}) \subset \bH(\overline{\KK})
 = \bC(\bG_{\rm ad}, r_{\rm DJ})(\overline{\KK}).$$
 This means that any twisted Belavin--Drinfeld cocycle $X$ for $r_{\rm BD}$ is simultaneously
 a twisted Belavin--Drinfeld cocycle for $r_{\rm DJ}$.

 As explained above, $X$ is equivalent to $J$, but this means by definition that $X=QJD$ for some $Q\in \bG_{\rm ad} (\KK)$ and $D\in \bH(\overline{\KK})$.
\end{proof}

Our next aim is to find  necessary and sufficient conditions for $D$ such that
$QJD$ is a twisted cocycle
for $r_{\rm BD}$.

\begin{proposition}\label{Y}
  $X=QJD\in \overline{Z}(\bG_{\rm ad},r_{\rm BD})$ if and only if the following two inclusions hold:
  \begin{enumerate}
  \item $D^{-1} \gamma_2 (D)\in \bC(\bG_{\rm ad}, r_{\rm BD})(\overline{\KK})$,
  \item $D^{-1} \gamma_1 (SDS)\in \bC(\bG_{\rm ad}, r_{\rm BD})(\overline{\KK})$.
  \end{enumerate}
\end{proposition}

\begin{proof}
  Assume that $X=QJD$ is a twisted cocycle for $r_{\rm BD}$.
  Then the first statement is clear because $\gamma_2 (QJ)=QJ$.

  Let us prove the second one. By definition we have
  $X^{-1} \gamma_1 (X) =SC=(SCS)S$ for some $C\in \bC(\bG_{\rm ad}, r_{\rm BD})(\overline{\KK})$.
  On the other hand,
$$
X^{-1} \gamma_1 (X) =D^{-1} J^{-1} \gamma_1 (J)\gamma_1 (D)=D^{-1} S\gamma_1(D)=D^{-1} (S\gamma_1(D)S)S.
$$
Hence, $SCS=D^{-1} S\gamma_1 (D)S \in \bC(\bG_{\rm ad}, r_{\rm BD})(\overline{\KK})$ because the centralizer is
stable under action of ${\rm Ad}_S$.

Conversely, consider $Y=QJD$, where $D$ satisfies conditions of the proposition. Then $Y^{-1} \gamma_2 (Y) = D^{-1} \gamma_2 (D)$ and
$\gamma_2 (Y)=Y D^{-1} \gamma(D)=YC$.

As for $\gamma_1$, we have to prove that $Y=QJD$ satisfies $\gamma_1 (QJD)=QJDSC$ for some
$C\in \bC(\bG_{\rm ad}, r_{\rm BD}).$ Since $\gamma_1 (Q)=Q$, it suffices to prove that
$\gamma_1 (JD)=JDSC$. We have
$$
\gamma_1 (JD)=JS\gamma_1 (D)=JDD^{-1}S\gamma_1 (D)=JDS(SD^{-1}S\gamma_1 (D)).
$$
Therefore, it remains to prove that $SD^{-1}S\gamma_1 (D)\in \bC(\bG_{\rm ad}, r_{\rm BD})$.
Since
$$D^{-1} \gamma_1 (SDS)\in \bC(\bG_{\rm ad}, r_{\rm BD}),$$
then
$${\rm Ad}_S (D^{-1} \gamma_1 (SDS))=SD^{-1}S\gamma_1 (D)\in \bC(\bG_{\rm ad}, r_{\rm BD}),$$
because the centralizer is ${\rm Ad}_S$-invariant.

This relation implies that $Y$ is a twisted cocycle for $r_{\rm BD}$.
\end{proof}

Now we discuss necessary and sufficient conditions for two twisted Belavin--Drinfeld cocycles
$X=Q_1 JD_1$ and $Y=Q_2 JD_2$ to be equivalent, namely, when $Y=Q_3 XC$, where $Q_i\in \bG_{\rm ad} (\KK)$, $i=1,2,3$, and
$D_j$, $j=1,2$, satisfy conditions of Proposition \ref{Y},
and $C\in \bC(\bG_{\rm ad}, r_{\rm BD})(\overline{\KK})$.

\begin{theorem}\label{R1}
  Let $X$ and $Y$ be two equivalent  twisted Belavin--Drinfeld cocycles for $r_{\rm BD}$. Then there exists $C \in \bC(\bG_{\rm ad}, r_{\rm BD})(\overline{\KK})$ for which the following two conditions hold:
\begin{enumerate}
 \item $D_1^{-1} \gamma_2 (D_1)=D_2^{-1} \gamma_2 (D_2 ) C^{-1} \gamma_2 (C)$,
 \item $D_1^{-1} \gamma_1 (SD_1 S)=D_2^{-1} \gamma_1 (SD_2 S) C^{-1} \gamma_1 (SCS)$.
\end{enumerate}
\end{theorem}

\begin{proof}
  Assume that $X,Y$ are two equivalent twisted Belavin--Drinfeld cocycles. Then
  $Y= Q_3 JD_1 C$ for some $C \in \bC(\bG_{\rm ad}, r_{\rm BD})(\overline{\KK})$ and $Q_3 \in \bG_{\rm ad} (\KK)$.
  It follows that
  \begin{itemize}
   \item $Y^{-1} \gamma_2 (Y)=C^{-1} D_1^{-1} \gamma_2 (CD_1)$,
   \item $Y^{-1} \gamma_1 (Y)=C^{-1} D_1^{-1} S\gamma_1 (CD_1)$.
  \end{itemize}
  On the other hand,
  \begin{itemize}
    \item $Y^{-1} \gamma_2 (Y)=D_2^{-1} \gamma_2 (D_2)$,
    \item $Y^{-1} \gamma_1 (Y)=D_2^{-1} S\gamma_1 (D_2)$.
  \end{itemize}
  An easy comparison of the equalities above completes the proof.
\end{proof}

Now we are motivated to introduce the following twisted action of $\mathcal{G}$ on $\bH (\overline{\KK})$.
There is a unique group homomorphism
$$u_S : \mathcal{G} \to \bG_{\rm ad}(\overline{\KK}) \subset \bAut({\bG}_{\rm ad})(\overline{\KK})$$
such that $u_S : \gamma_1 \mapsto \Ad_S.$ Since $S^2 = 1$, this homomorphism is continuous.
Furthermore, since $\mathcal{G}$ acts trivially on $\Ad_S$, our map $u_S$ is a cocycle in
$Z^1(\mathcal{G}, \bAut({\bG}_{\rm ad}))(\overline{\KK})$.
Since  $\bH(\overline{\KK})$  is stable under $\Ad_S$, we can consider
the corresponding twisted  
$\KK$-algebraic group $\bH_{u_S}$
and its Galois cohomology $H^1(\KK, \bH_{u_S})$.
Recall that by the definition of the twisted action (which we denote by $\ast$)
\begin{equation}\label{twistedaction}
 \gamma \ast D = u_S(\gamma)\big( \gamma(D)\big).
\end{equation}
In our case, for $D \in \bH(\overline{\KK})$ the explicit nature of the twisted action is
$$\gamma_2 \ast D =\gamma_2 (D) $$
and
 $$\gamma_1 \ast D=S\gamma_1 (D)S=\gamma_1 (SDS).$$
Similar considerations can be applied to  $\bC(\bG_{\rm ad}, r_{\rm BD})(\overline{\KK})$,
because this group is also $\Ad_S$-stable by Corollary \ref{C}.
The corresponding twisted $\KK$-group will be denoted by $\bC(\bG_{\rm ad}, r_{\rm BD})_{u_S}$.
Now, Proposition \ref{Y} can be reformulated.

\begin{proposition}
 Let $Q\in \bG_{\rm ad} (\KK)$ and $D\in \bH(\overline{\KK})$. Then
 $X=QJD \in \overline{Z}(\bG_{\rm ad},r_{\rm BD})$ if and only if
 $ D^{-1}(\sigma  \ast D) \in \bC(\bG_{\rm ad}, r_{\rm BD})(\overline{\KK})$ for all $\sigma \in \mathcal{G}$.\qed
\end{proposition}

Theorem \ref{R1} can be reformulated too.

\begin{theorem}\label{R}
  Let $X=Q_1 JD_1$ and $Y=Q_2 JD_2$ be two equivalent  twisted Belavin--Drinfeld cocycles for $r_{\rm BD}$.
  Then there exists $C \in \bC(\bG_{\rm ad}, r_{\rm BD})(\overline{\KK})$ such that
$D_1^{-1} (\sigma\ast D_1) =D_2^{-1} (\sigma\ast D_2 )  C^{-1} (\sigma\ast C )$ for any $\sigma \in \text{\rm Gal}(\KK)$.\qed
\end{theorem}

We need one more result.

\begin{theorem}\label{R2}
  Let $D_{1}, D_2\in \bH(\overline{\KK})$ be as in Proposition \ref{Y}. Let us assume that
  there exists $C\in \bC(\bG_{\rm ad}, r_{\rm BD})(\overline{\KK})$ such that
  $D_2^{-1} (\sigma\ast D_2)=D_1^{-1} (\sigma\ast D_1) C^{-1} (\sigma\ast C)$ for all
  $\sigma \in \text{\rm Gal}({\KK})$. Then for any $Q_1 , Q_2 \in \bG_{\rm ad}({\KK})$
  the elements $X=Q_1 JD_1$
  and $Y=Q_2 JD_2$ are equivalent as twisted Belavin--Drinfeld cocycles.
\end{theorem}

\begin{proof}
Clearly, it is sufficient to prove that $X_1 = JD_1$ and $Y_1 =JD_2$ are equivalent Belavin--Drinfeld cocycles. In other words, we have to prove that $Y_1 =QX_1 C_1$ for some $Q\in \bG_{\rm ad}({\KK})$ and $C_1 \in \bC(\bG_{\rm ad}, r_{\rm BD})(\overline{\KK})$.

We have $Y_1 =JD_2 =(JD_1 )(D_1^{-1}D_2 C^{-1})C$. By the conditions of the theorem, $D=D_1^{-1}D_2 C^{-1}$ satisfies  $\sigma\ast D =D$. Then we claim that $JD=QJ$ for some $Q\in \bG_{\rm ad} (\KK)$.

Let us prove this claim. It follows immediately that $D\in \bH (\LL)$ because $\sigma\ast D=\sigma (D)=D$ for all $\sigma \in \text{\rm Gal}({\LL})$. Further, $\gamma_1 (JD)=JS\gamma_1 (D)SS=J(S\gamma_1 (D)S)S=JDS$ because $S\gamma_1(D) S=D$. Taking into account that $\gamma_1 (J)=JS$ we obtain
$$ \gamma_1 (JDJ^{-1}) = (JDS)(SJ^{-1} )=JDJ^{-1}.$$
Hence, $JDJ^{-1} =Q\in \bG_{\rm ad} (\KK)$.

Finally, $Y_1 =JD_2 =(JD_1 )DC=Q(JD_1 )C=QX_1 C$.
\end{proof}

Theorems \ref{R1} and \ref{R2} mean that two twisted Belavin--Drinfeld cocycles
$X=Q_1 JD_1$
and $Y=Q_2 JD_2$ $r_{\rm BD}$  are equivalent if and only if
$$u_{D_1} (\gamma)=D_1^{-1}(\gamma\ast D_1)\ {\rm and}\ u_{D_2}(\gamma ) =D_2^{-1}(\gamma\ast D_2)$$
induce one and the same element in $H^1 (\KK ,\bC (\bG_{\rm ad}, r_{\rm BD})_{u_S})$.

\begin{corollary}
The map $w(QJD)= u_D$ defines an injective map from the set  $\overline{H} (\bG_{\rm ad}, r_{\rm BD})$ to $H^1 (\KK ,\bC(\bG_{\rm ad}, r_{\rm BD})_{u_S})$.
\end{corollary}

\begin{proof}
We need to show that the map
$$w: \overline{H} (\bG_{\rm ad}, r_{\rm BD}) \to H^1 (\KK ,\bC(\bG_{\rm ad}, r_{\rm BD})_{u_S})$$
is well defined and injective. To be precise, we need to show that if $Q_1JD_1$ and $Q_2JD_2$ are equivalent
Belavin--Drinfeld cocycles, then there exists $C \in \bC (\bG_{\rm ad}, r_{\rm BD})(\overline{\KK})$
such that for all $\sigma \in \mathrm{Gal}(\KK )$ we have
$$u_{D_1} (\sigma)= C^{-1} u_{D_2}(\sigma ) (\sigma\ast C).$$
In other words,
$$D_1^{-1}(\sigma\ast D_1) = C^{-1}  D_2^{-1}(\sigma \ast D_2)(\sigma \ast C). $$
But this is exactly Theorem \ref{R}.
\end{proof}

Our next aim is to prove

\begin{proposition}
The map $w$ is surjective.
\end{proposition}

\begin{proof}
We have an exact sequence of algebraic $\KK$-groups
$$ 1\to  \bC(\bG_{\rm ad}, r_{\rm BD})_{u_S}\to \bH_{u_S}$$
obtained by twisting the closed immersion $ \bC(\bG_{\rm ad}, r_{\rm BD}) \to \bH$.
Let $v$ be a  cocycle in $Z^1(\KK, {\bC(\bG_{\rm ad}, r_{\rm BD}})_{u_S})$. The image of $v$ in  $H^1 (\KK , \bH_{u_S} )$
is trivial since   $H^1 (\KK , \bH_{u_S})$ is trivial by a theorem of Steinberg (Serre Conjecture I), see \cite{Se} and \cite[page 185]{Stern}.
Thus, there exists $D \in \bH(\overline{\KK})$ such that $v(\gamma) = D^{-1} (\gamma \ast D)$ for all $\gamma \in \mathcal{G}$. Then $JD$
is a twisted Belavin--Drinfeld cocycle in  $\overline{Z}(\bG_{\rm ad}, r_{\rm BD})$ and $w(X) = u_D = v$.  This shows that $w$ is surjective.
\end{proof}

\begin{corollary}\label{BDG}
The map $w$ provides a bijection of sets
  $$\overline{H}(\bG_{\rm ad}, r_{\rm BD})\to H^1 (\KK ,\bC(\bG_{\rm ad}, r_{\rm BD})_{u_S}).$$\qed
\end{corollary}

The corollary generalizes (for our particular base field $\KK$) one of the main results of \cite{PS}
about non-twisted Belavin--Drinfeld cohomology. Namely, there exists a bijection of  sets $H(\bG_{\rm ad}, r_{\rm BD})\to H^1 (\KK ,\bC (\bG_{\rm ad}, r_{\rm BD}))$.

Our final result in this section is
the following  theorem,
which compares twisted and non-twisted
Belavin--Drinfeld cohomology.

\begin{theorem}\label{thm_twisted_nontwiset_relation}
  Assume that $r_{\rm BD}$ satisfies the conclusions of Proposition \ref{X} (i.e.\ the set
  $\overline{H}(\bG_{\rm ad}, r_{\rm BD})$ is non-empty).
  Then the set $\overline{H}(\bG_{\rm ad}, r_{\rm BD})$ is finite and its number of elements
  does not exceed the number of elements in ${H}(\bG_{\rm ad}, r_{\rm BD})$.
\end{theorem}

\begin{proof}
Since $\bC(\bG_{\rm ad}, r_{\rm BD})$ is a closed subgroup of $\bH$, it is of the form
$$
\bC(\bG_{\rm ad}, r_{\rm BD}) = \bT \times \bmu_{m_1} \times \ldots \times \bmu_{m_n},
$$
where $\bT$ is a split torus over $\KK$ and $\bmu_{m_k}$ is the finite multiplicative $\KK$-group of $m_k$-roots of unity.
Thus,
$$
H^1 (\KK, \bC(\bG_{\rm ad}, r_{\rm BD})) = \mathbb{K}^{\times}/ (\mathbb{K}^{\times})^{m_1}\times \ldots \times \mathbb{K}^{\times}/ (\mathbb{K}^{\times})^{m_n}=\Z/(m_1)\times\ldots\times\Z/(m_n).
$$
We consider  $r_{\rm BD}$ satisfying  conclusions of Proposition \ref{X}.
It is clear that the subtorus $\bT$ is stable under the action of $\mathrm{Ad}_S$.

Therefore, we can consider the following exact sequence of the twisted $\KK$-groups:
$$
1\to \bT_{u_S} \to \bC(\bG_{\rm ad}, r_{\rm BD})_{u_S}\to\bC(\bG_{\rm ad}, r_{\rm BD})_{u_S}/\bT_{u_S}  \to 1 .
$$
The last group in the sequence above is a twisted form of the finite constant group corresponding to $\Z/(m_1)\times\ldots\times\Z/(m_n)$.
Let us denote this $\KK$-group by
$\mathbf M$.

Now, consider
$$
H^1 (\KK, \bT_{u_S}) \to H^1 (\KK, \bC(\bG_{\rm ad}, r_{\rm BD})_{u_S})\to H^1 (\KK, \mathbf M).
$$
Since $\bT_{u_S}$ is reductive, we obtain $H^1 (\KK, \bT_{u_S})=\{ 1 \}$ by Steinberg's theorem
and consequently we get an embedding
$H^1 (\KK ,\bC(\bG_{\rm ad}, r_{\rm BD})_{u_S})\to H^1 (\KK, \mathbf M)$.

Let us estimate the number of elements of $H^1 (\KK ,\mathbf M)$.
Any element of $Z^1 (\mathbf M)$ is uniquely defined by the image of $\gamma_1$ in $\mathbf M (\KK)$
because $\rm{Gal}(\KK )$ is pro-cyclic and $\gamma_1$ is its pro-generator. Therefore,
$Z^1 (\mathbf M)$ contains $m_1 \ldots m_n$ elements, and the number of elements of $H^1 (\KK, \mathbf M)$, and thus of
$\overline{H}(\bG_{\rm ad}, r_{\rm BD})\simeq H^1 (\KK ,\bC(\bG_{\rm ad}, r_{\rm BD})_{u_S})$, is at most $m_1 \ldots m_n$.

As we have seen, ${H}(\bG_{\rm ad}, r_{\rm BD})\simeq H^1 (\KK, \bC(\bG_{\rm ad}, r_{\rm BD}))$ has exactly $m_1 \ldots m_n$ elements. This completes the proof.
 \end{proof}

\begin{corollary}
  Assume that $r_{\rm BD}$ satisfies the conclusions of Proposition \ref{X} and
  $H^1 (\KK ,\bC(\bG_{\rm ad}, r_{\rm BD}))=\{ 1\}$. Then $\overline{H}(\bG_{\rm ad}, r_{\rm BD})$
  consists of one element, which is $J$.\qed
\end{corollary}

\section{Classification of Lie bialgebras}\label{sec_classif_alg}



Let 
$\bG$ be a split simple algebraic group over any field $\FF$ of characteristic zero, $\bH\subset\bG$ a Cartan subgroup, $Q\subset P$ the root and weight lattices.
Let $\chi(\bH)$ be the group of (algebraic) characters of the torus $\bH$. The map $\lambda\mapsto \mathrm{d}\lambda$,
where $\mathrm{d}$ is the differential at the identity, is an isomorphism of $\chi(\bH)$ onto a lattice $X$ with $Q\subset X\subset P$.

Let $\gamma_1,\ldots,\gamma_n$ be a $\mathbb Z$-basis of $X$, $t_1,\ldots,t_n\in\chi(\bH)$ the corresponding characters.
Then the map $h\mapsto (t_1(h),\ldots, t_n(h))$ defines an isomorphism $\bH\to(\mathbb{G}_m)^n$ of algebraic tori.\footnote{By definition, $t_i\in\mathrm{Hom}(\bH,\mathbb{G}_m)$. Here $h\in\bH(R)$ for any ring extension $R\supset\FF$.}

\begin{proposition}\label{prop-conn-centr}
  Let $X=Q$, i.e.\ the group $\bG$ is of adjoint type.
  Then $\mathbf{C}(\bG, r_{\rm BD})$ is connected for any Belavin--Drinfeld r-matrix $r_{\rm BD}$.
\end{proposition}

\begin{proof}
Let the discrete parameter of $r_{\rm BD}$ be $(\Gamma_1, \Gamma_2, \tau)$. It follows from \cite[Theorem 2]{KKPS2} that $\mathbf{C}(\bG, r_{\rm BD})$ consists of all $h\in \bH $ such that for any $\alpha\in \Gamma_1$ we have
$e^{\alpha} (h)=e^{\tau (\alpha) }(h)$. Here $e^{\alpha}$ is the character of $\bH $ related to the
simple root $\alpha$.

If $X=Q$, we can choose $\gamma_i=\alpha_i$, where $\alpha_i$ are simple roots.
Then the centralizer $\mathbf{C}(\bG_{\rm ad}, r_{\rm BD})\subset \bH\simeq(\mathbb{G}_m)^n$ is
defined by equations of the form $t_{i_1}=\ldots=t_{i_k}$ for any string
$$\{\alpha_{i_1},\alpha_{i_2} =\tau (\alpha_{i_1}), \ldots,\alpha_{i_k} =\tau^{k-1} (\alpha_{i_1})   \}$$
of the r-matrix $r_{\rm BD}$. Therefore, $\mathbf{C}(\bG_{\rm ad}, r_{\rm BD})\simeq(\mathbb{G}_m)^{n(r_{\rm BD})}$, where $n(r_{\rm BD})$ is the number of strings of $r_{\rm BD}$ (including
strings which consist of one element only, i.e.\ the corresponding $\alpha$ is not contained in  $\Gamma_1$).
\end{proof}

\begin{remark}
If the lattice $X$ is bigger than $Q$, then each $\alpha_i =\sum n_{ij} \gamma_j$ with $n_{ij} \in \mathbb{Z}$.
Let $\bG=\bG_X$ be the corresponding group and let
$h=(h_1,\ldots, h_n)\in \bC (\bG_X,r_{\rm BD})(R)$ for a ring extension $R\supset\FF$.
Let $\alpha_i \in \Gamma_1$ and $\tau (\alpha_i) =\alpha_k =\sum n_{km} \gamma_m$.
Then we get the following equation
on the elements $h_s$:
\begin{equation}\label{system_centr}
\prod_j h_j^{n_{ij}}   = \prod_m h_m^{n_{km}}.
\end{equation}
Consequently, we get a system of equations which might lead to
non-connected\-ness of $\mathbf{C}(\bG, r_{\rm BD})$
as it happened for
$\bG=\mathbf{SO}_{2n}$, see \cite{KKPS3}.
See also Appendix \ref{appendix_B} with computations for $E_6$ and $E_7$.
\end{remark}

By \cite[Remark 4.11 and Corollary 4.13]{PS} we have

\begin{corollary}\label{cor_nontwisted_adjoint}
Let the base field $\FF$ be of cohomological dimension $1$ (eg., $\FF=\KK$). If $\bG$ is of adjoint type, then $H(\bG, r_{\rm BD})=\{1\}$ for any Belavin--Drinfeld r-matrix $r_{\rm BD}$ with $r_0\in\mathfrak{h} \otimes_\FF\mathfrak{h}$.\qed
\end{corollary}

Therefore, by Theorem \ref{thm_twisted_nontwiset_relation} we have

\begin{corollary}\label{cor_twisted_adjoint}
  Let the base field be $\KK$. Assume that $\bG$ is of adjoint type. Then $\overline{H}(\bG, r_{\rm BD})=\{J\}$
  for any Belavin--Drinfeld r-matrix $r_{\rm BD}$ with $\overline{Z}(\bG, r_{\rm BD})$ non-empty.\qed
\end{corollary}

\begin{remark}\label{rem_twisted_classical}
Note that in the non-trivial classical cases Corollary \ref{cor_twisted_adjoint} also follows from the explicit calculation of twisted Belavin--Drinfeld cohomology obtained in \cite{KKPS2, KKPS3}. Namely:

1) Let $\fg$ be of type $A_{n-1}$, $n\geq 3$. Then it follows from results of \cite{KKPS2} that
$\bC(\bGL_n, r_{\rm BD})$ is a split
sub-torus of $\bGL_n $ for any $r_{\rm BD}$. Consequently, $\bC(\bPGL_n, r_{\rm BD})$ is a split sub-torus of
$\bPGL_n$ and $\overline{H}(\bPGL_n, r_{\rm BD})$ is either empty or contains one element $J$ if
$r_{\rm BD}$ satisfies the conclusions of Proposition \ref{X}.

2) Let $\fg$ be of type $D_n$ with $n$ odd and the vertices of the corresponding Coxeter--Dynkin diagram $\alpha_{n-1},\ \alpha_n$
  be such that $d(\alpha_{n-1})=\alpha_n$.
  It follows from results of  \cite{KKPS3} that
$r_{\rm BD}$
satisfies the conclusions of Proposition \ref{X} if and only if $\Gamma_1 =\{ \alpha_{n-1}\}$, $\Gamma_2=\{ \alpha_{n}\}$ or
$\Gamma_1 =\{ \alpha_{n-1}, \alpha_k \}$, $\Gamma_2=\{ \alpha_k ,\alpha_{n}\}$, $\tau(\alpha_{n-1})=\alpha_k$, $\tau(\alpha_{k})=\alpha_n$.
Then for the corresponding
$r_{\rm BD}$ its centralizer in $\mathbf{SO}_{2n}$ is isomorphic to $\bT \times \{\pm I\}$, where
$\bT$ is a split sub-torus. It is clear that for the corresponding adjoint group we have
$H^1 (\KK,\bC(\mathbf{SO}_{2n} /\{\pm I\}, r_{\rm BD}))=\{ 1\}$. Consequently, if $r_{\rm BD}$
satisfies the conclusions of Proposition \ref{X}, then $\overline{H} (\mathbf{SO}_{4p+2}/\{\pm I\} , r_{\rm BD})=\{ J\}$.
\end{remark}

We now return to our classification. Let $\fg(\KK)$ be as above, and $\bG$ the algebraic $\KK$-group of adjoint type corresponding to $\fg(\KK)$. By Corollary \ref{cor_nontwisted_adjoint}, for any Belavin--Drinfeld triple $(\Gamma_1, \Gamma_2, \tau)$ and a continuous Belavin--Drinfeld parameter $r_0$ we have a unique, up to $\bG$-equivalence, Lie bialgebra structure on $\fg(\KK)$ of non-twisted type. Namely, let $\mathfrak{R}$ be the set of all quadruples $(\Gamma_1, \Gamma_2, \tau, r_0)$, where $(\Gamma_1, \Gamma_2, \tau)$ is a Belavin--Drinfeld triple, and $r_0\in\mathfrak{h} \otimes_\KK\mathfrak{h}$ is a continuous Belavin--Drinfeld parameter.

\begin{theorem}
Up to $\bG$-equivalence, Lie bialgebra structures on $\fg(\KK)$ of non-twisted type 
are parameterized by $\mathfrak{R}$.\qed
\end{theorem}

Let us say that r-matrices $r$ and $r'$ (and the corresponding Lie bialgebras) are \emph{$\bAut(\fg)$-equivalent} if $r' = b \, \varphi(r)$, where $\varphi\in\Aut(\fg)$, $b\in\KK^\times$. In order to classify Lie bialgebras up to $\bAut(\fg)$-equivalence, we need to describe an action of ${\rm Out} (\fg)$ on $\mathfrak{R}$.
Let $d\in {\rm Out} (\fg)$. Clearly, $d$ acts on the Cartan subalgebra of
$\mathfrak{h}\subset\fg(\KK)$ as $d(h_{\alpha})= d([e_{\alpha}, e_{-\alpha}])=h_{d({\alpha})}$, where $\alpha$ is a simple root. Then there is a natural action of ${\rm Out} (\fg)$
on the set $\mathfrak{R}$ given by
$$
d(\Gamma_1, \Gamma_2, \tau, r_0)=(d(\Gamma_1), d(\Gamma_2), d\tau d^{-1}, d(r_0)).
$$
Thus, we have the following

\begin{theorem}
Up to $\bAut(\fg)$-equivalence, Lie bialgebra structures on $\fg(\KK)$ of non-twisted type 
are parameterized by
$\rm{Out} (\fg)\backslash \mathfrak{R}$.\qed
\end{theorem}

Let us pass to the twisted type now. Let $\overline{\mathfrak{R}}$ be the set of all quadruples $(\Gamma_1, \Gamma_2, \tau, r_0)$, where a Belavin--Drinfeld triple $(\Gamma_1, \Gamma_2, \tau)$ 
and a continuous Belavin--Drinfeld parameter $r_0$ satisfy the conclusions of Proposition \ref{X}. By Corollary \ref{cor_twisted_adjoint}, we have

\begin{theorem}
Up to $\bG$-equivalence, Lie bialgebra structures on $\fg(\KK)$ of twisted type 
are parameterized
by $\overline{\mathfrak{R}}$.\qed
\end{theorem}

Now we classify twisted Lie bialgebra structures on $\fg(\KK)$ 
up to $\bAut(\fg)$-equivalence.

\begin{theorem}\label{kaka}
 Up to $\bAut(\fg)$-equivalence, Lie bialgebra structures on $\fg(\KK)$ of twisted type 
 are parameterized by
 $\rm{Out} (\fg)\backslash\overline{\mathfrak{R}}$.
\end{theorem}

\begin{proof}
We have to prove that $d(r_0)$ satisfies the condition $$\gamma_1 (d(r_0))={\rm Ad}_S (d(r_0)^{21}),$$  while $\gamma_1 (r_0)={\rm Ad}_S (r_0)^{21}$.
It is sufficient to prove that $d$ commutes with $\gamma_1$, which is obvious, and with ${\rm Ad}_S$.

Case 1: the Chevalley involution $c$ is not inner. Let us recall that in this case  $S=cd$,  where $d$ is the only automorphism of the Coxeter--Dynkin diagram which has order $2$. Notice that $c$ commutes with $d$, see (\ref{split}). Then, clearly, $d$ commutes with ${\rm Ad}_S$.

Case 2: $c$ is inner. Then, by construction of $S$, we have $S=c$, see \cite{PS}, and ${\rm Ad}_S$ acts identically on discrete parameters because it acts as $-\mathrm{Id}$ on the Cartan subalgebra. This observation completes the proof.
\end{proof}

\begin{remark}
1) If the Chevalley involution $c$ is inner, then, by Proposition \ref{21discrete}, $\Gamma_1=\Gamma_2=\emptyset$ for any $(\Gamma_1, \Gamma_2, \tau, r_0)\in\overline{\mathfrak{R}}$. In other words, $$\overline{\mathfrak{R}}=\overline{\mathfrak{R}}_{\mathrm{DJ}}:=\{r_0\in\mathfrak{h}(\LL)\otimes_\LL\mathfrak{h}(\LL)\,:\,
r_{0}+r_{0}^{21}=\Omega_{0}, \gamma_1 (r_0) = {\rm Ad}_S (r_0)^{21}\}.$$
Therefore, in this case Lie bialgebra structures on $\fg(\KK)$ of twisted type 
are parameterized by $\overline{\mathfrak{R}}_{\mathrm{DJ}}$ up to $\bG$-equivalence and by $\rm{Out}(\fg)\backslash\overline{\mathfrak{R}}_{\mathrm{DJ}}$ up to $\bAut(\fg)$-equivalence.

\smallskip

2) Let the Chevalley involution $c$ be outer. In this case we have $|\rm{Out} (\fg)|=2$, and $d\in\rm{Out} (\fg)$ of order $2$ acts on $\overline{\mathfrak{R}}$ by
\begin{equation}\label{d-action}
d(\Gamma_1, \Gamma_2, \tau, r_0)=(\Gamma_2, \Gamma_1, \tau^{-1}, d(r_0)).
\end{equation}
Therefore, in this case Lie bialgebra structures on $\fg(\KK)$ of twisted type 
are parameterized up to $\bAut(\fg)$-equivalence by $\overline{\mathfrak{R}}$ modulo the relation (\ref{d-action}).

\smallskip

The Chevalley involution is outer if and only if $\fg$ is of type $A_{n+1}, D_{2n+1}, E_6$. For the $A_{n+1}$ and $D_{2n+1}$ cases, see Remark \ref{rem_twisted_classical}. For the $E_6$ case, see Appendix \ref{appendix_B}.
\end{remark}



\section[Classification of quantum groups]{Classification of quantum groups\footnote{In this and the following sections we consider algebraic groups over $\mathbb C$.}}\label{sec_classif_gr}

According to \cite{EK1, EK2}, classification of quantum groups such that their
classical limit is $\fg(\KK)$ is equivalent to classification of Lie bialgebra structures on $\fg(\mathbb O)=\fg\otimes_{\mathbb C}\mathbb O$, where $\mathbb O=\mathbb C[[t]]$.

First recall \cite{KKPS2} that any Lie bialgebra structure on $\fg(\mathbb O)$ can be naturally extended to $\fg(\KK)$. Conversely, for any Lie bialgebra structure $\delta$ on $\fg(\KK)$ there exists a non-negative integer $n$ such that $t^n f(t)\delta$ for any invertible element $f(t)\in \OO^\times$ can be restricted onto $\fg(\mathbb O)$ and defines a Lie bialgebra structure on it.

Let us start with the non-twisted case.

\begin{theorem}\label{class}
Let $r= a \, \mathrm{Ad}_{X}(r_{\rm BD})$ and $r'= a' \, \mathrm{Ad}_{X'}(r'_{\rm BD})$ be two r-matrices of non-twisted type defining Lie bialgebra structures on $\fg(\mathbb O)$. Write non-twisted Belavin--Drinfeld cocycles $X$ and $X'$ as $X=QD$, $X'=Q'D'$, where $Q, Q'\in \bG_{\rm ad} (\KK)$, $D\in \bC(\bG_{\rm ad}, r_{\rm BD})(\overline{\KK})$, $D'\in \bC(\bG_{\rm ad}, r'_{\rm BD})(\overline{\KK})$. Then $r$ and $r'$ define $\bG_{\rm ad}(\mathbb O)$-equivalent Lie bialgebra structures on $\fg(\mathbb O)$ if and only if the following conditions hold:

(1) $a=a'$,

(2) $r_{\rm BD}=r'_{\rm BD}$,

(3) $Q$ and $Q'$ are in the same double coset in $$\bG_{\rm ad}(\mathbb O)\backslash\bG_{\rm ad}(\mathbb K)/\bC(\bG_{\rm ad}, r_{\rm BD})(\KK).$$
\end{theorem}

\begin{proof}
Assume that $r$ and $r'$ define $\bG_{\rm ad}(\mathbb O)$-equivalent Lie bialgebra structures on $\fg(\mathbb O)$. Notice that since $r_{\rm BD}+r_{\rm BD}^{21}=a\Omega$, $r'_{\rm BD}+(r'_{\rm BD})^{21}=a'\Omega$, and $\Omega$ is invariant with respect to automorphisms of $\fg(\mathbb O)$, we have $a=a'$ and $r_{\rm BD}=r'_{\rm BD}$.

Further, let us study when $r= a \, \mathrm{Ad}_{X}(r_{\rm BD})$ and $r'= a \, \mathrm{Ad}_{X'}(r_{\rm BD})$ induce $\bG_{\rm ad}(\mathbb O)$-equivalent Lie bialgebra structures on $\fg(\mathbb O)$. This condition is equivalent to $X'=YXC$, where $Y\in\bG_{\rm ad}(\mathbb O)$, $C\in\bC(\bG_{\rm ad}, r_{\rm BD})(\overline{\KK})$. Therefore, we have $Q'=YQZ$, where $Z=DC(D')^{-1}\in\bC(\bG_{\rm ad}, r_{\rm BD})(\KK)$. Conversely, having $Q'=YQZ$ with $Y\in\bG_{\rm ad}(\mathbb O)$, $Z\in\bC(\bG_{\rm ad}, r_{\rm BD})(\KK)$, we define $C=D^{-1}ZD'\in\bC(\bG_{\rm ad}, r_{\rm BD})(\overline{\KK})$ and obtain $X'=YXC$.
\end{proof}

\begin{remark}
The theorem above means that the quantum groups are parameterized by two parameters:
\begin{itemize}
\item {\it a continuous parameter} $a=t^n f(t)$,
\item {\it a double coset} in $\bG_{\rm ad}(\mathbb O)\backslash\bG_{\rm ad}(\mathbb K)/\bC(\bG_{\rm ad}, r_{\rm BD})(\KK)$. This parameter is discrete for $\mathfrak{sl}(2)$ and is not discrete
    already for $\mathfrak{sl}(3)$ as we will see later.
\end{itemize}
Since $\mathrm{Aut }(\fg(\mathbb O))$ is a semi-direct product of $\bG_{\rm ad}(\mathbb O)$ and a finite group $\mathrm{Out} (\fg )$, up to isomorphism quantum groups are classified by the continuous parameter $a=t^n f(t)$ and the set
    $$
    \mathrm{Out}(\fg)\backslash(\bG_{\rm ad}(\mathbb O)\backslash\bG_{\rm ad}(\mathbb K)/\bC(\bG_{\rm ad}, r_{\rm BD})(\KK)).
    $$
The action of $\mathrm{Out} (\fg )$ can be easily described: clearly $\mathrm{Out} (\fg )$   acts canonically on the simply connected
$\bG_{\rm sc}(\KK )$ and the action preserves the center, so it acts on $\bG_{\rm ad}(\mathbb K)$.
\end{remark}

Consider the case $\fg=\mathfrak{sl}(n)$ and $r_{\rm BD}=r_{\rm DJ}$. Notice that the natural projection $\mathbf{GL}(n, \mathbb K)\to\mathbf{PGL}(n, \mathbb K)$ induces a bijection $$\mathbf{GL}(n, \mathbb O)\backslash\mathbf{GL}(n, \mathbb K)/\mathbf{Diag}(n, \KK)\stackrel{\sim}{\rightarrow}\mathbf{PGL}(n, \mathbb O)\backslash\mathbf{PGL}(n, \mathbb K)/\mathbf{H}(\KK).$$

Let us discuss the set $\mathbf{GL}(n, \mathbb O)\backslash\mathbf{GL}(n, \mathbb K)/\mathbf{Diag}(n, \KK)$ for small values of $n$.

\begin{proposition}\label{gauss}
The set of representatives of $\mathbf{GL}(2, \mathbb O)\backslash\mathbf{GL}(2, \mathbb K)/\mathbf{Diag}(2, \KK)$ is $\left\{ T_i=\left(\begin{array}{cc}
1 & t^{-i} \\
0 & 1 \\
\end{array}\right) : i=0, 1, 2, \ldots\right\}$.
\end{proposition}

\begin{proof}
Using considerations similar to the Gauss algorithm we can conclude that any double coset
has a representative of the form above. Let us prove that they are distinct in the set of
double cosets. Indeed, let $T_i=PT_k H$, with $P\in \mathbf{GL}(2, \mathbb O)$ and $H$ diagonal.
It follows that $P$ is upper triangular and hence has the form
$P=\left(\begin{array}{cc}
y_1 & p \\
0 & y_2 \\
\end{array}\right)$ with invertible elements $y_i \in\OO^\times$ and $p\in \OO$. Furthemore,
we see that $H=\mathrm{diag}\,(y_1^{-1}, y_2^{-1})$. Multiplying, we get $p=y_2 t^{-k}-y_1 t^{-i}$.
Recall that $p\in \OO$. This can never happen unless $i=k$. 
\end{proof}

From the above proof also follows

\begin{corollary}\label{cor-Gauss}
Let $P,T_i , H$ be as above. If $PT_i H=T_i$, then
$$P=\left(\begin{array}{cc}
y_1 & (y_2-y_1)t^{-i}\\
0 & y_2 \\
\end{array}\right),
$$
$H=\mathrm{diag}\,(y_1^{-1}, y_2^{-1})$, and $y_1\equiv y_2 \pmod{t^i}$.\qed
\end{corollary}

\begin{proposition}
Representatives of $\mathbf{GL}(3, \mathbb O)\backslash\mathbf{GL}(3, \mathbb K)/\mathbf{Diag}(3, \KK)$
can be chosen of the form
$T_{ij}(q)=\left(\begin{array}{ccc}
1 & t^{-i} & q(t^{-1})\\
0 & 1 & t^{-j} \\
0 & 0 & 1 \\
\end{array}\right)$, where $i, j =0, 1, 2, \ldots$ and $q$ is a polynomial such that $q(0)=0$.
\end{proposition}

\begin{proof}
One can apply a Gauss type algorithm.
\end{proof}

It follows from Proposition \ref{gauss} and Corollary \ref{cor-Gauss} that if $T_{ij} (q_1)$ and
$T_{kl} (q_2)$ are contained in the same double coset, then $i=k, j=l$.
Furthermore, it follows that if $PT_{ij}(q_1)H=T_{ij}(q_2)$, then
$$
P=P_{ij}(y_1,y_2,p)=\left(\begin{array}{ccc}
y_1 & t^{-i}(y_2-y_1)& p\\
0 & y_2 & t^{-j}(1-y_2) \\
0 & 0 & 1 \\
\end{array}\right)
$$
and $H=\mathrm{diag}\,(y_1^{-1},y_2^{-1},1)$ with
$p\in \OO$, $y_1, y_2\in \OO^\times$ such that $y_2\equiv 1 \pmod{t^j}$,
$y_2\equiv y_1\pmod{t^i}$.

Let $f(t)=\sum_{-N}^M a_s t^s\in\KK$. Define $[f]=\sum_{-N}^{-1} a_s t^s$.

\begin{theorem}
$T_{ij} (q_1)$ and $T_{ij} (q_2)$ are in the same double coset if and only if
$q_2=[y_1 q_1+(y_2-y_1)t^{-i-j}]$ for some $y_1,y_2\in\OO^\times$ such
that $y_2\equiv 1\pmod{t^j}$ and $y_2\equiv y_1\pmod{t^i}$.
\end{theorem}

\begin{proof}
Calculating the product $P_{ij}(y_1,y_2,p)\cdot T_{ij}(q_1)\cdot\mathrm{diag}\,(y_1^{-1},y_2^{-1},1)$,
we can get positive degrees of $t$ in the upper right corner only.
Now, applying an elementary row operation we can ``kill'' the polynomial part in the upper right corner.
\end{proof}

So, we have constructed an action of the group
$$
N_{ij}=\{ (y_1,y_2)\in\OO^\times\times\OO^\times:\ y_2\equiv 1\!\!\pmod{t^j},\ \
y_2\equiv y_1\!\!\pmod{t^i}\}
$$
on the set of polynomials $P_0=t\CC [t]$.

\begin{corollary}
Double cosets $\mathbf{GL}(3, \mathbb O)\backslash\mathbf{GL}(3, \mathbb K)/\mathbf{Diag}(3, \KK)$
are in a bijection with the following data:

(1) A pair of non-negative integers $i,j$.

(2) An orbit of the action of the group $N_{ij}$ on the set $P_0$.\qed
\end{corollary}

\begin{remark}
A description of these orbits is an elementary problem, which we leave to the readers.
One can check that the orbit of the zero polynomial $N_{ij}(0)$ consists of all polynomials of degree $\leq j$.
\end{remark}

Now let us turn to the twisted case.

\begin{theorem}
Let $r= aj \, \mathrm{Ad}_{X}(r_{\rm BD})$ and $r'= a'j \, \mathrm{Ad}_{X'}(r'_{\rm BD})$ be two r-matrices of twisted type defining Lie bialgebra structures on $\fg(\mathbb O)$. Write twisted Belavin--Drinfeld cocycles $X$ and $X'$ as $X=QJD$, $X'=Q'JD'$, where $Q, Q'\in \bG_{\rm ad} (\KK)$, $D\in \bC(\bG_{\rm ad}, r_{\rm BD})(\overline{\KK})$, $D'\in \bC(\bG_{\rm ad}, r'_{\rm BD})(\overline{\KK})$. Then $r$ and $r'$ define $\bG_{\rm ad}(\mathbb O)$-equivalent Lie bialgebra structures on $\fg(\mathbb O)$ if and only if the following conditions hold:

(1) $a=a'$,

(2) $r_{\rm BD}=r'_{\rm BD}$,

(3) $Q$ and $Q'$ are in the same double coset in $$\bG_{\rm ad}(\mathbb O)\backslash\bG_{\rm ad}(\mathbb K)/(J\cdot\bC(\bG_{\rm ad}, r_{\rm BD})(\overline{\KK})\cdot J^{-1})\cap\bG_{\rm ad}(\KK).$$
\end{theorem}

\begin{proof}
Similar to the proof of Theorem \ref{class}.
\end{proof}

If $C\in (J\cdot\bC(\bG_{\rm ad}, r_{\rm BD})(\overline{\KK})\cdot J^{-1})\cap\bG_{\rm ad}(\KK)$,
then
$JCJ^{-1}=\gamma_1 (JCJ^{-1})=JS\gamma_1 (C)SJ^{-1}$.
Therefore, $\gamma_1 (C)=SCS $. Consequently,  $\gamma_2 (C)=C$ and
$C\in \bC(\bG_{\rm ad}, r_{\rm BD})({\LL})$.

Let us concentrate on the case $\fg=\mathfrak{sl} (2)$.
It is easy to show that in this case $C=\mathrm{diag}\,(d, \gamma_1 (d))$ with $d\in \LL$.
Another easy remark is that in this case we can substitute
$$\bG_{\rm ad}(\mathbb O)\backslash\bG_{\rm ad}(\mathbb K)/(J\cdot\bC(\bG_{\rm ad}, r_{\rm BD})(\overline{\KK})\cdot J^{-1})\cap\bG_{\rm ad}(\KK)$$ by
$$
\mathbf{GL}(2, \mathbb O)\backslash\mathbf{GL}(2, \mathbb K)/J\mathbf{D}_2 J^{-1}
\cong J\mathbf{D}_2 J^{-1}\backslash\mathbf{GL}(2, \mathbb K)/\mathbf{GL}(2, \mathbb O)
\cong
$$
$$\mathbf{D}_2\backslash J^{-1}\mathbf{GL}(2, \mathbb K)J /J^{-1}\mathbf{GL}(2, \mathbb O)J,$$
where $\mathbf{D}_2=\{\mathrm{diag}\,(d,\gamma_1 (d)): \ d\in\LL \}$ and 
$J=\left(\begin{array}{cc}
1 & 1\\
-j & j \\
\end{array}\right)$, see \cite{KKPS2}.
\vskip0.2cm
To study the latter set we need the theory of orders, see \cite{Rein}. The description is given in the appendix below.

\appendix
\section{
Double cosets and orders (by Juliusz Brzezinski and A. Stolin)}\label{appendix_A}

\subsection{Double cosets and orders in $\KK^n$}
In this subsection, we consider $\KK^n$ as a $\KK$-algebra with $\KK$ embedded diagonally
into $\KK^n$. Our purpose is to describe the double cosets which we discussed in the preceding section
in terms of $\OO$-orders in the algebra $\KK^n$.

\begin{definition} An $\OO$-module $M\subset \KK^n$ is called a \emph{lattice} on $\KK^n$ if its rank over $\OO$
is equal to $n$.
\end{definition}

Clearly, $\mathbf{GL}(n, \mathbb K)$ acts transitively on the set of lattices in $\KK^n$ because
any lattice has a form $M=G\cdot\OO^n$ for some $G  \in \mathbf{GL}(n, \mathbb K)$.
Hence, $$\mathbf{GL}(n, \mathbb K)/\mathbf{GL}(n, \mathbb O)\cong \{\mathrm{lattices}\ \mathrm{in}\ \KK^n\}.$$

\begin{definition} An \emph{order} in $\KK^n$ is a subring $\Lambda$ of  $\KK^n$ containing $\OO$, finitely generated
as an $\OO$-module and such that $\Lambda\KK = \KK^n$.
\end{definition}

\begin{definition}
Let $M\subset \KK^n$ be a lattice. Then $I(M)=\{x\in \KK^n:\ xM\subseteq M \}$
is called the \emph{set of multipliers} of $M$.
\end{definition}

The following lemma is well known.

\begin{lemma}
1) Any order $\Lambda$ is contained in $\OO^n$.

2) For any order $\Lambda$, we have $I(\Lambda)=\Lambda$.

3) For any lattice $M$, $I(M)$ is an order.\qed
\end{lemma}

\begin{proposition}
There is a canonical surjection
\begin{gather*}
\mathbf{GL}(n, \mathbb O)\backslash\mathbf{GL}(n, \mathbb K)/\mathbf{Diag}(n, \KK)\cong
\mathbf{Diag}(n, \KK)\backslash\mathbf{GL}(n, \mathbb K)/\mathbf{GL}(n, \mathbb O)\to\\
\{\mathrm{orders}\ \mathrm{in}\ \KK^n \}.
\end{gather*}
\end{proposition}

\begin{proof}
Consider two lattices in $\KK^n$, $M_1=G\cdot \OO^n$ and $M_2=H\cdot G\cdot \OO^n$, where $H\in \mathbf{Diag}(n, \KK)$.
Clearly, multiplication by $H=\mathrm{diag}\,(a_1,\ldots,a_n)$ coincides with multiplication by
$h=(a_1,\ldots,a_n)\in\KK^n$. Let $g\in I(M_1)$. Since the ring $\KK^n$ is commutative, it
follows that $g\in I(M_2)$ and so, by symmetry we have $I(M_1)=I(M_2)$. Thus, the correspondence $G\mapsto I(M_1)$ defines the required
map
$$\omega_n: \mathbf{Diag}(n, \KK)\backslash\mathbf{GL}(n, \mathbb K)/\mathbf{GL}(n, \mathbb O)\to
\{\mathrm{orders}\ \mathrm{in}\ \KK^n \}.$$
It is a surjection because for any order $\Lambda$ we have $I(\Lambda)=\Lambda$.
\end{proof}

Generally speaking, the map defined above is not injective. Let us define its kernel
in the sense of sets. More exactly, for any order $\Lambda$ we will find the subset
$\omega_n^{-1} (\Lambda)$.

\begin{definition}
Given an order $\Lambda$, we say that a lattice $M$ \emph{belongs} to $\Lambda$ if $\Lambda=I(M)$.
\end{definition}

It is clear that $M$ and $h\cdot M$, $h\in\KK^n$, belong to the same order $\Lambda=I(M)$.

\begin{definition}
We say that two lattices $M_1$ and $M_2$ are \emph{in the same lattice class} of $\Lambda$ if $M_1=hM_2$
for some $h\in\KK^n$.
\end{definition}

Let us consider a canonical map $\omega: \{\mathrm{lattices}\ \mathrm{in}\ \KK^n\}\to\{\mathrm{orders}\ \mathrm{in}\ \KK^n\}$ defined as  $M\mapsto I(M)$.
The following proposition is obvious.

\begin{proposition}
$\omega (M_1)=\omega (M_2)$ if $M_1$ and $M_2$ belong to the same lattice class.\qed
\end{proposition}

\begin{remark}
We remind readers that $\mathbf{GL}(n, \mathbb K)/\mathbf{GL}(n, \mathbb O)\cong \{\mathrm{lattices}\ \mathrm{in}\ \KK^n\}$.
Therefore, we can define a map
$$
d: \{\mathrm{lattices}\ \mathrm{in}\ \KK^n\}\to
\mathbf{Diag}(n, \KK)\backslash\mathbf{GL}(n, \mathbb K)/\mathbf{GL}(n, \mathbb O)
$$
and it is easy to see
that $\omega (M)=\omega_n (d(M))$. Moreover,
$$
\mathbf{Diag}(n, \KK)\backslash\mathbf{GL}(n, \mathbb K)/\mathbf{GL}(n, \mathbb O)\cong \bigcup_{\Lambda\subset \OO^n} \  \{\mathrm{lattice\ classes}\ \mathrm{belonging}\ \mathrm{to}\ \Lambda \}.
$$
\end{remark}

Let us fix an order $\Lambda$ and consider the set of lattices $L(\Lambda)$ belonging to $\Lambda$.
If $M\in L(\Lambda)$, then $hM\in L(\Lambda)$. Therefore, $L(\Lambda)$ is a disjoint union of lattice classes.
Let us denote the number of such classes by $\mathrm{lc}(\Lambda)$.
The number $\mathrm{lc}(\Lambda)$ is finite
because $\Lambda$ is finitely generated over $\OO$, which is a
discrete valuation ring.
The following proposition is a direct consequence of the remark above.

\begin{proposition}
$
\omega_n^{-1}(\Lambda)
=\{\mathrm{lattice\ classes}\ \mathrm{belonging}\ \mathrm{to}\ \Lambda \}
$
and hence,
$\omega_n^{-1}(\Lambda)$ consists of $\mathrm{lc}(\Lambda)$ elements.\qed
\end{proposition}

The result below was proved by J.~Brzezinski in \cite{BRZ}.

\begin{theorem}\label{thm_Brz}
$\mathrm{lc}(\Lambda)=1$ if and only if $\Lambda$ is a Gorenstein ring.\qed
\end{theorem}

\subsection{Quantum groups over $\mathfrak{sl}(2)$}
We begin with a corollary to Brzezinski's theorem.

\begin{corollary}
The map $\omega_2$ is a bijection.
\end{corollary}

\begin{proof}
Let $\Lambda$ be an order in $\OO^2$. Then it is of the form
$\Lambda=\OO [y]$, where $y$ satisfies a quadratic equation
$y^2+ay+b=0$ with $a,b\in\OO$. It is known that such a ring
is Gorenstein. Therefore, $\mathrm{lc}(\Lambda)=1$ and $\omega_2$ is a bijection.
\end{proof}

\begin{proposition}
Any order $\Lambda\subset\OO^2$ is a free $\OO$-module $\Lambda_n$ with a basis
$\{(1,1), (t^n,0)\}$, $n=0,1,\ldots$. The orders $\Lambda_{n_1}$ and $\Lambda_{n_2}$ are not isomorphic if $n_1\neq n_2$
and hence, quantum groups of non-twisted type over $\mathfrak{sl}(2)$ are
indexed by non-negative integers.
\end{proposition}

\begin{proof}
Let $\Lambda$ have a basis $\{(1,1), (a,b)\}$ with $a,b\in \OO$. Then $\{(1,1), (a-b,0)\}$
is also a basis. Let $a-b=xt^n$, where $x\in\OO$ is invertible and $n$ is a non-negative integer number.
Therefore, $\{(1,1), (t^n,0)\}$ is a basis. The rest is clear.
\end{proof}

Let us also discuss the twisted case, in other words the double cosets
$$\mathbf{D}_2\backslash J^{-1}\mathbf{GL}(2, \mathbb K)J /J^{-1}\mathbf{GL}(2, \mathbb O)J,$$
where $\mathbf{D}_2=\{\mathrm{diag}\,(d,\gamma_1 (d)): \ d\in\LL\}$.

The lemma below is straightforward.

\begin{lemma}
$J^{-1}\mathbf{GL}(2, \mathbb K)J =\mathbf{U}(1,1)$.\qed
\end{lemma}
Here, in an analogy with the real numbers, we denote by $\mathbf{U}(1,1)$ the group which consists of matrices of the form
$$
P=\left(\begin{array}{cc}
x & y\\
\gamma_1 (y) & \gamma_1 (x)\\
\end{array}\right)
$$
with $x,y\in \LL$.

The group $\mathbf{U}(1,1)$ acts naturally on $\LL$ via the formula $Pd=xd+y\gamma_1 (d)$.
In fact, this action comes from the natural action of $\mathbf{U}(1,1)$ on $\LL^2$
and the embedding $\LL \to \LL^2$, $d\mapsto (d,\gamma_1 (d))$.

Now we can repeat the non-twisted considerations above.

\begin{definition}
$M\subset \LL$ is a \emph{lattice} in $\LL$ if it is an $\OO$-submodule of $\LL$ of  rank $2$.
\end{definition}

It is not difficult to show that
$$
J^{-1}\mathbf{GL}(2, \mathbb K)J /J^{-1}\mathbf{GL}(2, \mathbb O)J\cong \{\mathrm{lattices}\ \mathrm{in}\ \LL \}.
$$

\begin{definition}
  $\Lambda\subset \LL$ is an \emph{order} in $\LL$ if it is a lattice and a sub-ring of $\LL$ which contains
  the unit of $\LL$.
\end{definition}

\begin{remark}
One can show that in fact $\Lambda\subset \OO_\LL =\OO [j]$.
\end{remark}

Using the result by Brzezinski \cite{BRZ}, we deduce the final classification of the twisted quantum groups for $\mathfrak{sl}(2)$.

\begin{theorem}
There is a canonical bijection
$$\rho: \mathbf{D}_2\backslash J^{-1}\mathbf{GL}(2, \mathbb K)J /J^{-1}\mathbf{GL}(2, \mathbb O)J\to \{\mathrm{orders}\ \mathrm{in}\ \LL  \} = \{\OO [t^{n+\frac{1}{2} }] :\  n\in {\mathbb Z}_+\}.$$
\end{theorem}

\begin{proof}
  As in the non-twisted case, we have to show that any order $\Lambda$ in $\LL$ is a Gorenstein ring, which is clear because it can be easily shown that in this case
  $\Lambda=\OO [t^{n+\frac{1}{2}}]$, $n\in {\mathbb Z}_+$.
\end{proof}

\begin{corollary}
Quantum groups such that their classical limit is $\mathfrak{sl}(2)$ are in a one-to-one correspondence with the set of orders in separable quadratic rings, i.e.\ $\OO^2$ and $\OO_\LL$. The corresponding orders were described above.\qed
\end{corollary}

\subsection{Quantum groups over $\mathfrak{sl} (3)$ and orders in cubic rings}
The aim of this subsection is to classify  quantum groups such that their classical limit is
$\mathfrak{sl}(3)$ with the Lie bialgebra structure defined by
$r_{\rm DJ}$ and $jr_{\rm DJ}$
in terms of cubic rings.

Our considerations are based on results about orders in cubic rings obtained in \cite{DF} and
\cite{DKF}, see also \cite{Br} and \cite{GL}.
We begin with the non-twisted case.
If $n=3$, the bijection
$$
\mathbf{Diag}(3, \KK)\backslash\mathbf{GL}(3, \mathbb K)/\mathbf{GL}(3, \mathbb O)\cong \bigcup_{\Lambda\subset \OO^3} \  \{\mathrm{lattice\ classes}\ \mathrm{belonging}\ \mathrm{to}\ \Lambda \}.
$$
has been already constructed.

Let us turn to the twisted case. We have
$$
J=J_{3}=\left(\begin{array}{ccc}
1 & 0 & 1\\
0 & 1 & 0 \\
-j & 0 & j \\
\end{array}\right)
$$
(see \cite{KKPS2}). Because of this particular form of $J_3$, our treatment of the case $n=3$ is very similar to the case $n=2$.

Let us present an element of $\LL\oplus\KK$ in the form
$(x,a,\gamma_1 (x))$, where $x\in\LL$, $a\in\KK$. Then, it is clear
that there is a bijection of sets
$$
J_3^{-1}\mathbf{GL}(3, \mathbb K)J_3 /J_3^{-1}\mathbf{GL}(3, \mathbb O)J_3
\cong\{\rm{lattices}\ \rm{in}\ \LL\oplus\KK\}.
$$
Let us define $\mathbf{D}_3=\{\mathrm{diag}\,(d,a,\gamma_1 (d)): d\in\LL,a\in\KK\}$.

Let $N$ be a lattice in $\LL\oplus\KK$. Let us define the \emph{ring of multipliers}
of $N$ as $I(N)=\{x\in\mathbf{D}_3: \ xN\subset N\}$. Clearly,
$I(N)\subset\OO_\LL\oplus\OO$ is an order.
The following result takes place.

\begin{theorem}
1) Quantum groups of the twisted type which quantize the Lie bialgebra structure on
$\mathfrak{sl}(3)$ defined by $jr_\mathrm{DJ}$ are parameterized by
$$\mathbf{D}_3\backslash J_3^{-1}\mathbf{GL}(3, \mathbb K)J_3 /J_3^{-1}\mathbf{GL}(3, \mathbb O)J_3$$

2) There is a natural surjection
$$
\rho_3: \mathbf{D}_3\backslash J_3^{-1}\mathbf{GL}(3, \mathbb K)J_3 /J_3^{-1}\mathbf{GL}(3, \mathbb O)J_3
\to \{\mathrm{orders}\ \mathrm{in}\ \OO_{\LL}\oplus\OO\}.
$$\qed
\end{theorem}

Given an order $\Lambda\subset \OO_\LL\oplus\OO$,
we say that a lattice $N$ \emph{belongs} to $\Lambda$ if $I(N)=\Lambda$.
Further, we say that two lattices $N_1$ and $N_2$ are \emph{in the same
lattice class} if $N_2=xN_1$ for some $x\in \LL\oplus\KK$.
Clearly, $I(N_1)=I(N_2)$ and the set of lattices belonging to $\Lambda$
is a disjoint union of lattice classes.

\begin{corollary}
\begin{gather*}
\mathbf{D}_3\backslash J_3^{-1}\mathbf{GL}(3, \mathbb K)J_3 /J_3^{-1}\mathbf{GL}(3, \mathbb O)J_3\cong\\
\bigcup_{\Lambda\subset \OO_\LL\oplus\OO} \  \{\mathrm{lattice\ classes}\ \mathrm{belonging}\ \mathrm{to}\ \Lambda \}.
\end{gather*}\qed
\end{corollary}

Further, we need to study orders in cubic rings $\KK^3$ and $\LL\oplus\KK$. In the next two
subsections, we give two approaches to this description.

\subsection{Classification of cubic orders contained in separable cubic algebras I}

We begin with a general construction of cubic rings following \cite{DF},
see also \cite{Br}, \cite{GL}. Let $R$ be a discrete valuation ring (e.g., $R=\mathbb O$) and $K$ its quotient field.
Assume that any quadratic field extension of $K$ is generated by an element of $R$ whose square equals a generator of the
maximal ideal of $R$.
Let $A$ be a cubic separable $K$-algebra.
For every $R$-order $\Lambda$ in $A$, write $\Lambda= R+R\omega+R\theta$. Translating
$\omega$ and $\theta$ by appropriate elements of $R$, we can achieve that $\omega\theta=n\in R$.
Such a basis we will call {\it normal}. So, we got the following multiplication table:
\begin{equation*}
\omega\theta =n, \hspace{0,5cm} \omega^2= m +b\omega-a\theta, \hspace{0.5cm} \theta^2 =l+d\omega -c\theta,
\end{equation*}
where $a,b,c,d,l,m,n\in R$.
One can show that the associative law implies
that $(n,m,l)=(-ad,-ac,-bd)$, i.e., we get
\begin{equation}\label{Jast}
\omega\theta = -ad, \hspace{0,5cm} \omega^2= -ac +b\omega-a\theta, \hspace{0.5cm} \theta^2 =-bd + d\omega -c\theta.
\end{equation}

Now let us consider the {\it index form} $f(x,y)=ax^3+bx^2y+cxy^2+dy^3$ of $\Lambda$. Notice that the index form $f$ determines $\Lambda=\Lambda(f)=\Lambda_{abcd}$ uniquely
up to an isomorphism.

Let $P_\omega(X)=X^3 -bX^2 + acX- a^2d$ and $P_\theta(X)= X^3+cX^2+bdX+ad^2$.

\begin{lemma}\label{lm}
$P_\theta(\theta)=0$ and $P_\omega(\omega)=0$.
\end{lemma}

\begin{proof}
To derive the first equation, we multiply both sides of the third relation in the multiplication table above by $\theta$ and take into
account that $\omega\theta=n=-ad$. We get the second equation similarly.
\end{proof}

\begin{remark}\label{rem_P}
If $ad\neq 0$, then $P_{\theta}(-ad/X)=(ad^2 /X^3)P_{\omega}(X)$. If $a=1$, then
$P_{\omega} (X)=f(X,-1)$.
\end{remark}

\begin{theorem}
If $A=K\Lambda(f)$, then $A$ is a field if and only if  $P_\omega(X)$
 is irreducible over $K$.
\end{theorem}

\begin{proof}
Let $A$ be a field. Since $\omega \in A \setminus K$ is a zero of the polynomial $P_\omega(X)$ of degree $3$, this polynomial is minimal for $\omega$ over $K$. Thus, it is irreducible
over $K$.  Conversely, if $P_\omega(X)$ is irreducible over $K,$ then $K(\omega)$ is a field extension of degree $3$ over $K$, so $K(\omega)=A$.
\end{proof}

\begin{remark}
Clearly, if $P_\omega (X)$ is irreducible, then $P_\theta(X)$ is also irreducible
because irreducibility of $P_\omega(X)$ implies that $ad\neq 0$, and then we can use Remark \ref{rem_P}.
\end{remark}

As we know, if $A$ is a separable algebra of degree 3 over $K$, then $A$ is either
a (separable) field extension of $K$, or $A$ is isomorphic to a product of a quadratic
(separable) field extension $L$ of $K$ by $K$, or $A$ is isomorphic to $K^3$. If $A=K\Lambda(f)$, then we already know that $A$ is a field if and only if $P_\omega(X)$ (and $P_\theta(X)$) are irreducible.
Moreover, the algebra $A=K\Lambda(f)$ is separable if and only if the discriminant 

$$\Delta(f) =   18abcd + b^2 c^2-4ac^3-4db^3-27a^2 d^2 \neq 0.$$

Now we want to distinguish between the two remaining cases using the {\it index form} $f(X,Y)$.

We need the following auxiliary result:

\begin{lemma} The elements $1,\omega,\omega^2$ form a basis of $A$ over $K$ if and only if $a \neq 0$, while $1,\theta,\theta^2$ form a basis of $A$ if and only if $d \neq 0$.
\end{lemma}

\begin{proof}
Follows immediately from the relations (\ref{Jast})
taking into account that $1,\omega, \theta$ is a basis of $A$.
\end{proof}

\begin{proposition}
If $a \neq 0$ and the polynomial
$P_\omega(X)$ is reducible over $K$, then

$(a)$ $A$ is isomorphic to $L\oplus K$ if $P_\omega(X)$ has only one zero in $K$,

$(b)$  $A$ is isomorphic to $K^3$ if $P_\omega(X)$ has three $($different$)$ zeros in $K$.

The same is true when $d \neq 0$ and $P_\omega(X)$  is replaced by $P_\theta(X)$.
\end{proposition}

\begin{proof}
If $a \neq 0$, then by the lemma above, the elements $1,\omega,\omega^2$
generate $A$, which implies that $A \cong K[X]/(P_\omega(X))$ and both (a) and (b)
are evident. The same arguments work when $d \neq 0$ and $P_\omega(X)$  is replaced by $P_\theta(X)$.
\end{proof}

It remains the case when $a=d=0$. The multiplication rules (\ref{Jast}) reduce then to
$$\omega\theta =0, \hspace{0,5cm} \omega^2= b\omega, \hspace{0.5cm} \theta^2 =-c\theta.$$

Notice that 
$\Delta(f)=b^2c^2 \neq 0$, since $A$ is separable.

\begin{proposition}
If $a=d=0$ and $A=K\Lambda(f)$ is a separable algebra, then $A \cong K^3$.
\end{proposition}

\begin{proof}
It is easy to see that  $A \cong K[X,Y]/(X^2-bX,Y^2+cY,XY)$.  Since $A$ is separable, we have to exclude a possibility that $A$ contains
a quadratic field extension of $K$. In our case, such a quadratic field extension is generated by an element of $A$ whose square equals a generator $t$ of the maximal ideal of $R$. A general element of  $K[X,Y]/(X^2-bX,Y^2+cY,XY)$
has the form $\alpha+\beta x+\gamma y$, where $\alpha, \beta, \gamma \in K$ and $x^2=bx, y^2=-cy,xy =0$. Thus
$(\alpha+\beta x+\gamma y)^2 =t$ implies that $\alpha ^2=t$, where $\alpha \in K$, which is impossible.
\end{proof}

Notice that in the case $a=d=0$ the polynomials $P_\omega(X), P_\theta(X)$ have all their zeros in $K$.
Thus, we have

\begin{corollary}
$(a)$ The separable algebra $A=K \Lambda( f)$ is isomorphic to $K^3$ if and only if
both polynomials $P_{\omega}(X)$ and $P_{\theta}(X)$ have all their zeros in $K$.

$(b)$  The separable algebra $A=K \Lambda (f)$ is isomorphic to $L\oplus K$ if and only if at least one of
the polynomials $P_{\omega}(X)$ or $P_{\theta}(X)$ has a root in $L$.
\end{corollary}

Now return to the case $R=\OO$. We are almost ready to complete our description of double cosets (and therefore, our classification of the corresponding quantum groups) in terms of quadruples $(a,b,c,d)$.

First, we define an action of $\mathbf{GL}(2,\OO)$ on the set
of index forms and hence, on the set of quadruples $(a,b,c,d)$. Let $g\in\mathbf{GL}(2,\OO)$.
The action is defined as follows:
$$
f(u,v)\mapsto g\cdot f(u,v)=\frac{1}{\mathrm{det}(g)}f((u,v)g).
$$
Here, we consider $(u,v)$ as a row.

The result below was proved in \cite{GL}, see also \cite{Br} and \cite{DF}.

\begin{proposition}
Let $S$ be either a local ring or a principal ideal domain. Then there is a bijection
between the set of orbits  of the action of $\mathbf{GL}(2,S)$ on the set of index forms
(and hence, on the set of quadruples $(a,b,c,d)$) and the set of isomorphism classes
of cubic rings over $S$.\qed
\end{proposition}

Let us make the following observation:

\begin{lemma}
  Let $r:\Lambda\to \Lambda'$ be an $\OO$-algebra isomorphism. Then we can extend $r$ to a $\KK$-isomorphism $r': \KK\Lambda\to \KK\Lambda'$ of the corresponding enveloping algebras (and therefore, they are isomorphic).
\end{lemma}

\begin{proof}
  Clearly, $r$ can be extended to $r': \KK\Lambda\to \KK\Lambda'$ as $r' (a\otimes k)=r(a)\otimes k$.
\end{proof}

Let us denote the set of quadruples $(a,b,c,d)$ such that the corresponding cubic order is contained in $\KK^3$ (resp.\ $\LL\oplus\KK$) by $\cal{P}$ (resp.\ $\cal{Q}$).

\begin{corollary}
The sets $\cal{P}$, $\cal{Q}$ are invariant under the action of  $\mathbf{GL}(2,\OO)$.\qed
\end{corollary}

Let $\mathrm{Aut}_\KK (\KK\Lambda)$ be the group of $\KK$-automorphisms of the enveloping algebra $\KK\Lambda$.

\begin{corollary}
  There are two bijections of sets
  $$
  \mathrm{Aut}_\KK (\LL\oplus\KK)\backslash\{\mathrm{orders}\ \mathrm{in}\ \OO_\LL\oplus\OO\}\cong
  \mathbf{GL}( 2,\OO)\backslash{\cal Q},
  $$
  $$
  \mathrm{Aut}_\KK (\KK^3)\backslash\{\mathrm{orders}\ \mathrm{in}\ \OO^3\}\cong
  \mathbf{GL}(2,\OO)\backslash{\cal P}.
  $$
\end{corollary}

\begin{proof}
  It is sufficient to notice that any $\KK$-automorphism of the enveloping algebra preserves the corresponding
  maximal order, $\OO^3$ or $ \OO_\LL\oplus \OO$.
\end{proof}

\begin{remark}
It is easy to show that
$\mathrm{Aut}_\KK (\LL\oplus\KK)\cong \mathrm{Aut}_\KK (\LL)\cong\mathbb{Z}/2\mathbb{Z}$ and
$\mathrm{Aut}_\KK (\KK^3)\cong S_3$, the symmetric group.
\end{remark}

Now, we can describe the set of quantum groups related to the orders contained in $\KK^3$ as follows.
\begin{itemize}
  \item Choose a representative $(a,b,c,d)$ in $\mathbf{GL}(2,\OO)\backslash{\cal P}$.
  \item Construct $\Lambda_{abcd}\subset \KK^3$.
  \item Quantum groups corresponding to the orbit of the quadruple $(a,b,c,d)$ are in
  a one-to-one correspondence with lattices  in $\KK^3$ such that their ring of multipliers
  is $\gamma (\Lambda_{abcd})$, where $\gamma$ is an automorphism of $\KK^3$.
  \end{itemize}
The set of quantum groups related to the orders contained in $\KK\oplus\LL$ has an almost identical description.
\begin{example}
Assume that $ad\neq 0$ and $(a,b,c,d)\in{\cal P}$. The equation $P_{\theta}(x)=0$ has three roots
$x_1,x_2,x_3\in\OO$ and we can set $\theta=(x_1,x_2,x_3)\in\KK^3$. Then
$\omega=(-ad/x_1,-ad/x_2,-ad/x_3)$ and $\mathrm{Aut}_\KK (\KK^3)=S_3$ acts on $\Lambda_{abcd}$ as a permutation group. It is not necessary that all six orders $\gamma (\Lambda_{abcd}),\ \gamma\in S_3$
are distinct. It might happen that some of them coincide.
\end{example}

In order to complete our description of quantum groups in terms of quadruples,
we have to describe the set of lattice classes belonging to an order $\Lambda$
in terms of $a,b,c,d$.

The result below is a consequence of general results of \cite{DKF} applied to the
ring $\OO$.

\begin{theorem}
If $a,b,c,d\in t\OO$, then $\mathrm{lc}(\Lambda_{abcd})=2$.
Otherwise, $\mathrm{lc}(\Lambda_{abcd})=1$.\qed
\end{theorem}

\begin{remark}
Notice that, according to Theorem \ref{thm_Brz}, if $a,b,c,d\in t\OO$, then $\Lambda_{abcd}$ is not Gorenstein.
Otherwise, $\Lambda_{abcd}$ is Gorenstein.
\end{remark}

\subsection{Classification of cubic orders contained in separable  cubic algebras II}\label{cubic_orders}
Here, we give a different approach to the classification problem of cubic orders.
Again, let  $R$ be a discrete valuation ring (e.g., $R=\mathbb O$) and $K$ its quotient field. Denote by $t$ a generator of the maximal ideal of $R$. If $\Lambda' \subset \Lambda$ are two $R$-orders in a $K$-algebra $A$, then the product
of the invariant factors (see \cite[(4.14)]{Rein}) of this pair (of $R$-lattices) is a power of the ideal $(t)$.
We write $[\Lambda:\Lambda']=t^k$ if this product of the invariant factors is $(t^k)$ and we call $t^k$ or simply
$k$ for the index of $\Lambda'$ in $\Lambda$.

\medskip

{\bf Description of all $R$-orders in the $K$-algebra $K^3$.}

\smallskip

We consider the field $K$ as diagonally embedded into $K^3$. The maximal order in this algebra
is $\Lambda=R^3$. Choose as a basis of $R^3$ the following elements: $e_1=1 = (1,1,1)$, $e_2=(0,1,0)$
and $e_3=(0,0,1)$. Of course, we have $e_2^2=e_2, e_3^2=e_3$ and $e_2e_3=0$.
Let $\Lambda' \subset \Lambda$ be any $R$-suborder of $\Lambda$. Let $1,f_2,f_3$ be an $R$-basis
of $\Lambda'$. It is clear that $1$ always can be chosen as a part of such a basis
since $\Lambda'/R$ is torsion-free and $\Lambda'$ is $R$-projective (even free). This means
that we can choose $f_2 = \alpha e_2+\beta e_3, f_3=\gamma e_2+\delta e_3$, where $\alpha,\beta,\gamma,\delta \in R$.

Assume now that $\Lambda'$ is a Gorenstein order, that is, $\alpha,\beta,\gamma,\delta$ are relatively prime. Otherwise,
we have $\Lambda' = R + t\Lambda''$, where $\Lambda''$ is a suborder of $\Lambda$. When $\Lambda'$
is Gorenstein, at least one of $\alpha,\beta,\gamma,\delta$ is invertible in $R$, say $\alpha$, and we can assume that $\alpha=1$. Thus, we may choose $\gamma=0$, so that $f_3=\delta e_3$. Further, we may assume that $\delta=t^k$ for a nonnegative integer $k$. Since $\Lambda'$ is an order, we have $f_2^2, f_3^2, f_2f_3 \in \Lambda'$. Only the first condition puts some restrictions on $\beta$:
$$f_2^2= e_2+ \beta^2e_3$$
implies that there exist $k,l \in R$ such that $e_2+ \beta^2e_3 =k(e_2+\beta e_3)+lt^ke_3$.
Hence, we get $k=1$ and $\beta^2=\beta+lt^k$. The second equation shows that $\beta \equiv 0,1 \pmod {t^k}$.
Thus, we get two possibilities: $f_2=e_2, f_3= t^ke_3$ or $f_2=e_2+e_3, f_3=t^ke_3$. It is easy to check that the orders $\Lambda_k =R+Re_2+Rt^ke_3$ and $\Lambda'_k=R + R(e_2+e_3)+Rt^ke_3$ are Gorenstein and
different if only $k > 0$ (if $k=0$, we get the maximal order $\Lambda$).
Thus, we have proved the following

\begin{theorem}
For every index $[\Lambda:\Lambda']=t^k$, where $k >0$, we have exactly two Gorenstein suborders of $\Lambda=R^3$, namely $\Lambda_k$ and $\Lambda'_k$. All other proper suborders of $\Lambda$ are not Gorenstein and are
$\Lambda_{k,l}=R+t^l\Lambda_k$ and $\Lambda'_{k,l}=R+t^l\Lambda'_k$, where $k >0$, $l >0$.
The number of all suborders of $\Lambda$ of given index $n=k+2l$ equals $\left[ \frac{n}{2} \right]+1$, $n \geq 0$.\qed
\end{theorem}

\medskip

{\bf Description of all $R$-orders in the $K$-algebra $K \oplus L$, where $L$ is a quadratic field over $K$.}

\smallskip

Let $L=K(j)$, where $j^2=t$. We consider the field $K$ as embedded diagonally into $K \oplus L$. The maximal order in this algebra is $\Lambda = R \oplus S$, where $S$ is the maximal $R$-order in $L$. Choose as a basis of $R \oplus S$ the following elements: $e_1=1 = (1,1)$, $e_2=(0,1)$ and $e_3=(0,j)$. Of course, we have $e_2^2=e_2, e_3^2=te_2$ and $e_2e_3=e_3$.
Let $\Lambda' \subset \Lambda$ be any $R$-suborder of $\Lambda = R \oplus S$. Let $1,f_2,f_3$ be an $R$-basis
of $\Lambda'$. It is clear that $1$ always can be chosen as a part of basis of $\Lambda'$ for the same reasons as in the case of $\Lambda = R^3$. This means that we can choose $f_2 = \alpha e_2+\beta e_3, f_3=\gamma e_2+\delta e_3$, where $\alpha,\beta,\gamma,\delta\in R$.

Assume now that $\Lambda'$ is a Gorenstein order, that is, $\alpha,\beta,\gamma,\delta$ are relatively prime. Otherwise,
we have $\Lambda' = R + t\Lambda''$, where $\Lambda''$ is a suborder of $\Lambda$. When $\Lambda'$
is Gorenstein, at least one of $\alpha,\beta,\gamma,\delta$ is invertible in $R$.

{\bf Case I}. If one of $\alpha,\gamma$ is invertible in $R$, then without loss of generality we can
assume that $\alpha=1$. Thus, we may choose
$\gamma=0$, so that $f_3=\delta e_3$. Further, we may assume that $\delta=t^k$ for a nonnegative integer $k$.
Since $\Lambda'$ is a suborder, we have $f_2^2, f_3^2, f_2f_3 \in \Lambda'$. As before, only the first condition puts some restrictions on $\beta$:
$$f_2^2= e_2+2\beta e_3+ \beta^2te_2=(1+\beta^2)te_2+2\beta e_3$$
implies that there exist $k,l \in R$ such that $e_2+2\beta e_3+ \beta^2te_2=k(e_2+\beta e_3)+lt^ke_3$.
Hence, we get $k=1+\beta^2t$ and $2\beta= k\beta+lt^k$, which gives $2\beta=(1+\beta^2t)\beta+lt^k$. Thus, we have
$\beta \equiv 0 \pmod {t^k}$. As a consequence, we get that $\Lambda_k=R+Re_2+Rt^ke_3$ is the only Gorenstein suborder of $\Lambda$ of index $[\Lambda:\Lambda_k]=t^k$.
All other proper suborders of $\Lambda$ are not Gorenstein and are
$\Lambda_{k,l}=R+t^l\Lambda_k$, where $k >0$, $l >0$.
Observe also that in this case the number of all suborders of $\Lambda$ of given index $n=k+2l$ equals $\left[ \frac{n}{2} \right]+1$, $n \geq 0$.

{\bf Case II}. If $t$ divides both $\alpha, \gamma$ and one of $\beta,\delta$ is invertible in $R$, then without loss of generality we can assume that $\beta=1$. Thus, we may choose
$\delta=0$, so that $f_2=\alpha e_2+e_3$ and $f_3=\gamma e_2$. As before, since $\Lambda'$ is a suborder, we have $f_2^2, f_3^2, f_2f_3 \in \Lambda'$. We easily check that also this time only the first condition puts some
restrictions on the coefficients (this time $\alpha,\gamma$):
$$f_2^2= \alpha^2e_2+2\alpha e_3+ te_2=(\alpha^2+t)e_2+2\alpha e_3$$
implies that there exist $k,l \in R$ such that $(\alpha^2+t)e_2+2\alpha e_3 =k(\alpha e_2+e_3)+l\gamma e_2$.
Hence, we get $k=2\alpha$ and $\alpha^2+t=k\alpha+l\gamma$, which implies that $l\gamma=t-\alpha^2$. Since $t \mid \gamma$ and $t^2 \mid \alpha^2$, we get $l \in R$ only if $t^2 \nmid \gamma$. Hence, we can choose $f_2=e_3$ and $f_3= te_2$,
so $\Lambda'=R+Rte_2+Re_3$ is the only Gorenstein suborder of $\Lambda$ in this case.
The order $\Lambda'_k = R + t^k\Lambda'$ for integer $k >0$ is not Gorenstein and has index
$[\Lambda:\Lambda'_k]=t^{2k+1}$.

To summarize, we get the following

\begin{theorem}
The maximal order $\Lambda = R \oplus S = R+Re_2+Re_3$, where $e_2^2=e_2, e_3^2=te_2$ and $e_2e_3=e_3$
in $K \oplus L$ contains exactly one Gorenstein suborder $\Lambda_k =R+Re_2+Rt^ke_3$ of every index $k > 1$, while
for $k=1$, there are two Gorenstein suborders of index $1$, $\Lambda_1 =R+Re_2+Rte_3$ and $\Lambda'_1=R+Rte_2+Re_3$.
All non-Gorensteins suborders of $\Lambda$ are $\Lambda_{k,l}=R+t^l\Lambda_k$, where $k >0,l >0$ (of index $k+2l$) and
$\Lambda'_k = R + t^k\Lambda'_1$, where $k > 0$ (of index $2k+1$). The total number of suborders of $\Lambda$
of given index $n$ is equal $\left[ \frac{n}{2} \right]+1$ for even $n$ and $\left[ \frac{n}{2} \right]+2$ for odd $n$.\qed
\end{theorem}
\begin{remark}
At this point we would like remind the reader that in the case $R=\OO$ we have one quantum group corresponding to a Gorenstein order and two quantum groups which correspond to a non-Gorenstein order.
\end{remark}
\medskip
Our results are quite unexpected: there are ``too many'' quantum groups which are not isomorphic as Hopf algebras over $\OO$.
However, we make a conjecture that after tensoring by $\KK$ there will be only two Hopf algebras over $\KK$ related to
non-twisted and twisted Belavin--Drinfeld cohomology.

\section{Belavin--Drinfeld cohomology for exceptional simple Lie algebras (by E. Karolinsky and Aleksandra Pirogova)}\label{appendix_B}

In this appendix we discuss Belavin--Drinfeld cohomology for exceptional simple Lie algebras. We keep notation introduced in Section \ref{sec_classif_alg}. Let $\bG$ be a split simple simply connected (i.e., $X=P$) algebraic group of exceptional type. If $\bG$ is of type $G_2$, $F_4$, or $E_8$, then $P=Q$, i.e., $\bG$ is of adjoint type, and therefore, by Proposition \ref{prop-conn-centr}, the centralizer $\bC(\bG, r_{\rm BD})$ is connected for any Belavin--Drinfeld r-matrix $r_{\rm BD}$. The remaining cases are $E_6$ and $E_7$. In the $E_6$ case,
$\Gamma=\{\alpha_1,\ldots,\alpha_5,\alpha_6\}$ is enumerated in a way that $\{\alpha_1,\ldots,\alpha_5\}$ is the simple root system of type $A_5$ (with the standard enumeration).

\begin{theorem}\label{thm_E6}
1) In the $E_6$ case, the centralizer $\mathbf{C}(\bG, r_{\rm BD})$ is not connected if and only if one of the following (mutually non-exclusive) conditions hold: either $\alpha_1$ and $\alpha_2$ are in the same string and $\alpha_4$ and $\alpha_5$ are also in the same string, or $\alpha_1$ and $\alpha_5$ are in the same string and $\alpha_2$ and $\alpha_4$ are also in the same string. In these cases $\bC(\bG, r_{\rm BD}) = \bT \times \bmu_3$,
where $\bT$ is a split torus and $\bmu_3$ is the group of cubic roots of unity.

2) In the $E_7$ case, the centralizer $\mathbf{C}(\bG, r_{\rm BD})$ is connected for any Belavin--Drinfeld r-matrix $r_{\rm BD}$.
\end{theorem}

\begin{proof}
The proof is via brute force aided by a computer. Namely, first, using a program written in C++, we list all possible admissible triples and compute the corresponding strings. Then, using Wolfram Mathematica, in each case we solve the corresponding system of equations (\ref{system_centr}) and compute the centralizer.
\end{proof}

Applying \cite[Remark 4.11 and Corollary 4.13]{PS}, we get

\begin{corollary}
Let the base field $\FF$ be of cohomological dimension $1$. Let $r_{\rm BD}$ be a Belavin--Drinfeld r-matrix with $r_0\in\mathfrak{h} \otimes_\FF\mathfrak{h}$.

1) In the $E_6$ case, $H(\bG, r_{\rm BD})=\FF^\times/(\FF^\times)^3$ in the cases when $\bC(\bG, r_{\rm BD}) = \bT \times \bmu_3$. Otherwise, $H(\bG, r_{\rm BD})=\{1\}$.

2) In the $G_2$, $F_4$, $E_7$, and $E_8$ cases, $H(\bG, r_{\rm BD})=\{1\}$.\qed
\end{corollary}

For the $E_6$ case, totally there are $406=203\times2$ admissible triples (with non-empty $\Gamma_1$ and $\Gamma_2$). Among these, $70=35\times2$ triples satisfy the condition of Theorem \ref{thm_E6}. They are listed below (up to interchanging $\Gamma_1$ and $\Gamma_2$). First, we list the corresponding strings, and then the admissible triples having the given string structure.
\begin{itemize}
\item $\{\alpha_1, \alpha_2\}$, $\{\alpha_4, \alpha_5\}$
    \begin{itemize}
    \item[$\ast$] $\Gamma_1=\{\alpha_1, \alpha_4\}$, $\Gamma_2=\{\alpha_2, \alpha_5\}$, $\tau(\alpha_1)=\alpha_2$, $\tau(\alpha_4)=\alpha_5$;
    \item[$\ast$] $\Gamma_1=\{\alpha_1, \alpha_5\}$, $\Gamma_2=\{\alpha_2, \alpha_4\}$, $\tau(\alpha_1)=\alpha_2$, $\tau(\alpha_5)=\alpha_4$.
    \end{itemize}
\item $\{\alpha_1, \alpha_5\}$, $\{\alpha_2, \alpha_4\}$
    \begin{itemize}
    \item[$\ast$] $\Gamma_1=\{\alpha_1, \alpha_2\}$, $\Gamma_2=\{\alpha_5, \alpha_4\}$, $\tau(\alpha_1)=\alpha_5$, $\tau(\alpha_2)=\alpha_4$;
    \item[$\ast$] $\Gamma_1=\{\alpha_1, \alpha_4\}$, $\Gamma_2=\{\alpha_5, \alpha_2\}$, $\tau(\alpha_1)=\alpha_5$, $\tau(\alpha_4)=\alpha_2$.
    \end{itemize}
\item $\{\alpha_1, \alpha_2\}$, $\{\alpha_3, \alpha_4, \alpha_5\}$
    \begin{itemize}
    \item[$\ast$] $\Gamma_1=\{\alpha_1, \alpha_3, \alpha_4\}$, $\Gamma_2=\{\alpha_2, \alpha_4, \alpha_5\}$,\\ $\tau(\alpha_1)=\alpha_2$, $\tau(\alpha_3)=\alpha_4$, $\tau(\alpha_4)=\alpha_5$.
    \end{itemize}
\item $\{\alpha_1, \alpha_2, \alpha_3\}$, $\{\alpha_4, \alpha_5\}$
    \begin{itemize}
    \item[$\ast$] $\Gamma_1=\{\alpha_1, \alpha_2, \alpha_4\}$, $\Gamma_2=\{\alpha_2, \alpha_3, \alpha_5\}$,\\ $\tau(\alpha_1)=\alpha_2$, $\tau(\alpha_2)=\alpha_3$, $\tau(\alpha_4)=\alpha_5$.
    \end{itemize}
\item $\{\alpha_1, \alpha_5\}$, $\{\alpha_2, \alpha_3, \alpha_4\}$
    \begin{itemize}
    \item[$\ast$] $\Gamma_1=\{\alpha_1, \alpha_3, \alpha_4\}$, $\Gamma_2=\{\alpha_5, \alpha_2, \alpha_3\}$,\\ $\tau(\alpha_1)=\alpha_5$, $\tau(\alpha_3)=\alpha_2$, $\tau(\alpha_4)=\alpha_3$.
    \end{itemize}
\item $\{\alpha_1, \alpha_3, \alpha_5\}$, $\{\alpha_2, \alpha_4\}$
    \begin{itemize}
    \item[$\ast$] $\Gamma_1=\{\alpha_1, \alpha_2, \alpha_3\}$, $\Gamma_2=\{\alpha_3, \alpha_4, \alpha_5\}$,\\ $\tau(\alpha_1)=\alpha_3$, $\tau(\alpha_2)=\alpha_4$, $\tau(\alpha_3)=\alpha_5$;
    \item[$\ast$] $\Gamma_1=\{\alpha_1, \alpha_2, \alpha_5\}$, $\Gamma_2=\{\alpha_3, \alpha_4, \alpha_1\}$,\\ $\tau(\alpha_1)=\alpha_3$, $\tau(\alpha_2)=\alpha_4$, $\tau(\alpha_5)=\alpha_1$;
    \item[$\ast$] $\Gamma_1=\{\alpha_1, \alpha_4, \alpha_5\}$, $\Gamma_2=\{\alpha_5, \alpha_2, \alpha_3\}$,\\ $\tau(\alpha_1)=\alpha_5$, $\tau(\alpha_4)=\alpha_2$, $\tau(\alpha_5)=\alpha_3$.
    \end{itemize}
\item $\{\alpha_1, \alpha_2\}$, $\{\alpha_4, \alpha_5, \alpha_6\}$
    \begin{itemize}
    \item[$\ast$] $\Gamma_1=\{\alpha_1, \alpha_4, \alpha_6\}$, $\Gamma_2=\{\alpha_2, \alpha_6, \alpha_5\}$,\\ $\tau(\alpha_1)=\alpha_2$, $\tau(\alpha_4)=\alpha_6$, $\tau(\alpha_6)=\alpha_5$;
    \item[$\ast$] $\Gamma_1=\{\alpha_1, \alpha_5, \alpha_6\}$, $\Gamma_2=\{\alpha_2, \alpha_6, \alpha_4\}$,\\ $\tau(\alpha_1)=\alpha_2$, $\tau(\alpha_5)=\alpha_6$, $\tau(\alpha_6)=\alpha_4$.
    \end{itemize}
\item $\{\alpha_1, \alpha_2, \alpha_6\}$, $\{\alpha_4, \alpha_5\}$
    \begin{itemize}
    \item[$\ast$] $\Gamma_1=\{\alpha_1, \alpha_4, \alpha_6\}$, $\Gamma_2=\{\alpha_6, \alpha_5, \alpha_2\}$,\\ $\tau(\alpha_1)=\alpha_6$, $\tau(\alpha_4)=\alpha_5$, $\tau(\alpha_6)=\alpha_2$;
    \item[$\ast$] $\Gamma_1=\{\alpha_1, \alpha_5, \alpha_6\}$, $\Gamma_2=\{\alpha_6, \alpha_4, \alpha_2\}$,\\ $\tau(\alpha_1)=\alpha_6$, $\tau(\alpha_5)=\alpha_4$, $\tau(\alpha_6)=\alpha_2$.
    \end{itemize}
\item $\{\alpha_1, \alpha_5\}$, $\{\alpha_2, \alpha_4, \alpha_6\}$
    \begin{itemize}
    \item[$\ast$] $\Gamma_1=\{\alpha_1, \alpha_2, \alpha_4\}$, $\Gamma_2=\{\alpha_5, \alpha_4, \alpha_6\}$,\\ $\tau(\alpha_1)=\alpha_5$, $\tau(\alpha_2)=\alpha_4$, $\tau(\alpha_4)=\alpha_6$;
    \item[$\ast$] $\Gamma_1=\{\alpha_1, \alpha_2, \alpha_6\}$, $\Gamma_2=\{\alpha_5, \alpha_4, \alpha_2\}$,\\ $\tau(\alpha_1)=\alpha_5$, $\tau(\alpha_2)=\alpha_4$, $\tau(\alpha_6)=\alpha_2$;
    \item[$\ast$] $\Gamma_1=\{\alpha_1, \alpha_4, \alpha_6\}$, $\Gamma_2=\{\alpha_5, \alpha_6, \alpha_2\}$,\\ $\tau(\alpha_1)=\alpha_5$, $\tau(\alpha_4)=\alpha_6$, $\tau(\alpha_6)=\alpha_2$.
    \end{itemize}
\item $\{\alpha_1, \alpha_5, \alpha_6\}$, $\{\alpha_2, \alpha_4\}$
    \begin{itemize}
    \item[$\ast$] $\Gamma_1=\{\alpha_1, \alpha_2, \alpha_5\}$, $\Gamma_2=\{\alpha_5, \alpha_4, \alpha_6\}$,\\ $\tau(\alpha_1)=\alpha_5$, $\tau(\alpha_2)=\alpha_4$, $\tau(\alpha_5)=\alpha_6$;
    \item[$\ast$] $\Gamma_1=\{\alpha_1, \alpha_2, \alpha_6\}$, $\Gamma_2=\{\alpha_5, \alpha_4, \alpha_1\}$,\\ $\tau(\alpha_1)=\alpha_5$, $\tau(\alpha_2)=\alpha_4$, $\tau(\alpha_6)=\alpha_1$;
    \item[$\ast$] $\Gamma_1=\{\alpha_1, \alpha_4, \alpha_6\}$, $\Gamma_2=\{\alpha_6, \alpha_2, \alpha_5\}$,\\ $\tau(\alpha_1)=\alpha_6$, $\tau(\alpha_4)=\alpha_2$, $\tau(\alpha_6)=\alpha_5$.
    \end{itemize}
\item $\{\alpha_1, \alpha_2, \alpha_4, \alpha_5\}$
    \begin{itemize}
    \item[$\ast$] $\Gamma_1=\{\alpha_1, \alpha_2, \alpha_4\}$, $\Gamma_2=\{\alpha_4, \alpha_5, \alpha_2\}$,\\ $\tau(\alpha_1)=\alpha_4$, $\tau(\alpha_2)=\alpha_5$, $\tau(\alpha_4)=\alpha_2$;
    \item[$\ast$] $\Gamma_1=\{\alpha_1, \alpha_2, \alpha_4\}$, $\Gamma_2=\{\alpha_5, \alpha_4, \alpha_1\}$,\\ $\tau(\alpha_1)=\alpha_5$, $\tau(\alpha_2)=\alpha_4$, $\tau(\alpha_4)=\alpha_1$;
    \item[$\ast$] $\Gamma_1=\{\alpha_1, \alpha_2, \alpha_5\}$, $\Gamma_2=\{\alpha_4, \alpha_5, \alpha_1\}$,\\ $\tau(\alpha_1)=\alpha_4$, $\tau(\alpha_2)=\alpha_5$, $\tau(\alpha_5)=\alpha_1$;
    \item[$\ast$] $\Gamma_1=\{\alpha_1, \alpha_2, \alpha_5\}$, $\Gamma_2=\{\alpha_5, \alpha_4, \alpha_2\}$,\\ $\tau(\alpha_1)=\alpha_5$, $\tau(\alpha_2)=\alpha_4$, $\tau(\alpha_5)=\alpha_2$.
    \end{itemize}
\item $\{\alpha_1, \alpha_2\}$, $\{\alpha_3, \alpha_6\}$, $\{\alpha_4, \alpha_5\}$
    \begin{itemize}
    \item[$\ast$] $\Gamma_1=\{\alpha_1, \alpha_3, \alpha_5\}$, $\Gamma_2=\{\alpha_2, \alpha_6, \alpha_4\}$,\\ $\tau(\alpha_1)=\alpha_2$, $\tau(\alpha_3)=\alpha_6$, $\tau(\alpha_5)=\alpha_4$.
    \end{itemize}
\item $\{\alpha_1, \alpha_5, \alpha_6\}$, $\{\alpha_2, \alpha_3, \alpha_4\}$
    \begin{itemize}
    \item[$\ast$] $\Gamma_1=\{\alpha_1, \alpha_2, \alpha_3, \alpha_5\}$, $\Gamma_2=\{\alpha_6, \alpha_3, \alpha_4, \alpha_1\}$,\\
        $\tau(\alpha_1)=\alpha_6$, $\tau(\alpha_2)=\alpha_3$, $\tau(\alpha_3)=\alpha_4$, $\tau(\alpha_5)=\alpha_1$;
    \item[$\ast$] $\Gamma_1=\{\alpha_1, \alpha_2, \alpha_3, \alpha_6\}$, $\Gamma_2=\{\alpha_6, \alpha_3, \alpha_4, \alpha_5\}$,\\
        $\tau(\alpha_1)=\alpha_6$, $\tau(\alpha_2)=\alpha_3$, $\tau(\alpha_3)=\alpha_4$, $\tau(\alpha_6)=\alpha_5$;
    \item[$\ast$] $\Gamma_1=\{\alpha_1, \alpha_3, \alpha_4, \alpha_5\}$, $\Gamma_2=\{\alpha_5, \alpha_2, \alpha_3, \alpha_6\}$,\\
        $\tau(\alpha_1)=\alpha_5$, $\tau(\alpha_3)=\alpha_2$, $\tau(\alpha_4)=\alpha_3$, $\tau(\alpha_5)=\alpha_6$.
    \end{itemize}
\item $\{\alpha_1, \alpha_2\}$, $\{\alpha_3, \alpha_4, \alpha_5, \alpha_6\}$
    \begin{itemize}
    \item[$\ast$] $\Gamma_1=\{\alpha_1, \alpha_3, \alpha_5, \alpha_6\}$, $\Gamma_2=\{\alpha_2, \alpha_5, \alpha_6, \alpha_4\}$,\\
        $\tau(\alpha_1)=\alpha_2$, $\tau(\alpha_3)=\alpha_5$, $\tau(\alpha_5)=\alpha_6$, $\tau(\alpha_6)=\alpha_4$.
    \end{itemize}
\item $\{\alpha_1, \alpha_2, \alpha_3, \alpha_6\}$, $\{\alpha_4, \alpha_5\}$
    \begin{itemize}
    \item[$\ast$] $\Gamma_1=\{\alpha_1, \alpha_2, \alpha_4, \alpha_6\}$, $\Gamma_2=\{\alpha_3, \alpha_6, \alpha_5, \alpha_1\}$,\\
        $\tau(\alpha_1)=\alpha_3$, $\tau(\alpha_2)=\alpha_6$, $\tau(\alpha_4)=\alpha_5$, $\tau(\alpha_6)=\alpha_1$.
    \end{itemize}
\item $\{\alpha_1, \alpha_2, \alpha_3, \alpha_4, \alpha_5\}$
    \begin{itemize}
    \item[$\ast$] $\Gamma_1=\{\alpha_1, \alpha_2, \alpha_3, \alpha_4\}$, $\Gamma_2=\{\alpha_2, \alpha_3, \alpha_4, \alpha_5\}$,\\
        $\tau(\alpha_1)=\alpha_2$, $\tau(\alpha_2)=\alpha_3$, $\tau(\alpha_3)=\alpha_4$, $\tau(\alpha_4)=\alpha_5$.
    \end{itemize}
\item $\{\alpha_1, \alpha_2, \alpha_4, \alpha_5, \alpha_6\}$
    \begin{itemize}
    \item[$\ast$] $\Gamma_1=\{\alpha_1, \alpha_2, \alpha_4, \alpha_6\}$, $\Gamma_2=\{\alpha_4, \alpha_5, \alpha_6, \alpha_2\}$,\\
        $\tau(\alpha_1)=\alpha_4$, $\tau(\alpha_2)=\alpha_5$, $\tau(\alpha_4)=\alpha_6$, $\tau(\alpha_6)=\alpha_2$;
    \item[$\ast$] $\Gamma_1=\{\alpha_1, \alpha_2, \alpha_4, \alpha_6\}$, $\Gamma_2=\{\alpha_5, \alpha_4, \alpha_6, \alpha_1\}$,\\
        $\tau(\alpha_1)=\alpha_5$, $\tau(\alpha_2)=\alpha_4$, $\tau(\alpha_4)=\alpha_6$, $\tau(\alpha_6)=\alpha_1$;
    \item[$\ast$] $\Gamma_1=\{\alpha_1, \alpha_2, \alpha_5, \alpha_6\}$, $\Gamma_2=\{\alpha_4, \alpha_5, \alpha_6, \alpha_1\}$,\\
        $\tau(\alpha_1)=\alpha_4$, $\tau(\alpha_2)=\alpha_5$, $\tau(\alpha_5)=\alpha_6$, $\tau(\alpha_6)=\alpha_1$;
    \item[$\ast$] $\Gamma_1=\{\alpha_1, \alpha_2, \alpha_5, \alpha_6\}$, $\Gamma_2=\{\alpha_5, \alpha_4, \alpha_6, \alpha_2\}$,\\
        $\tau(\alpha_1)=\alpha_5$, $\tau(\alpha_2)=\alpha_4$, $\tau(\alpha_5)=\alpha_6$, $\tau(\alpha_6)=\alpha_2$.
    \end{itemize}
\end{itemize}

We also list the admissible triples that satisfy the conclusions of Proposition \ref{X}. There are $40=20\times2$ such triples (with non-empty $\Gamma_1$ and $\Gamma_2$). Their list (up to interchanging $\Gamma_1$ and $\Gamma_2$) is given below.
\begin{itemize}
\item $\Gamma_1=\{\alpha_1\}$, $\Gamma_2=\{\alpha_5\}$, $\tau(\alpha_1)=\alpha_5$;
\item $\Gamma_1=\{\alpha_2\}$, $\Gamma_2=\{\alpha_4\}$, $\tau(\alpha_2)=\alpha_4$;
\item $\Gamma_1=\{\alpha_1, \alpha_2\}$, $\Gamma_2=\{\alpha_4, \alpha_5\}$, $\tau(\alpha_1)=\alpha_4$, $\tau(\alpha_2)=\alpha_5$;
\item $\Gamma_1=\{\alpha_1, \alpha_2\}$, $\Gamma_2=\{\alpha_5, \alpha_4\}$, $\tau(\alpha_1)=\alpha_5$, $\tau(\alpha_2)=\alpha_4$;
\item $\Gamma_1=\{\alpha_1, \alpha_3\}$, $\Gamma_2=\{\alpha_3, \alpha_5\}$, $\tau(\alpha_1)=\alpha_3$, $\tau(\alpha_3)=\alpha_5$;
\item $\Gamma_1=\{\alpha_1, \alpha_4\}$, $\Gamma_2=\{\alpha_2, \alpha_5\}$, $\tau(\alpha_1)=\alpha_2$, $\tau(\alpha_4)=\alpha_5$;
\item $\Gamma_1=\{\alpha_1, \alpha_4\}$, $\Gamma_2=\{\alpha_5, \alpha_2\}$, $\tau(\alpha_1)=\alpha_5$, $\tau(\alpha_4)=\alpha_2$;
\item $\Gamma_1=\{\alpha_1, \alpha_6\}$, $\Gamma_2=\{\alpha_6, \alpha_5\}$, $\tau(\alpha_1)=\alpha_6$, $\tau(\alpha_6)=\alpha_5$;
\item $\Gamma_1=\{\alpha_2, \alpha_3\}$, $\Gamma_2=\{\alpha_3, \alpha_4\}$, $\tau(\alpha_2)=\alpha_3$, $\tau(\alpha_3)=\alpha_4$;
\item $\Gamma_1=\{\alpha_2, \alpha_6\}$, $\Gamma_2=\{\alpha_6, \alpha_4\}$, $\tau(\alpha_2)=\alpha_6$, $\tau(\alpha_6)=\alpha_4$;
\item $\Gamma_1=\{\alpha_1, \alpha_2, \alpha_3\}$, $\Gamma_2=\{\alpha_3, \alpha_4, \alpha_5\}$,\\ $\tau(\alpha_1)=\alpha_3$, $\tau(\alpha_2)=\alpha_4$, $\tau(\alpha_3)=\alpha_5$;
\item $\Gamma_1=\{\alpha_1, \alpha_2, \alpha_4\}$, $\Gamma_2=\{\alpha_4, \alpha_5, \alpha_2\}$,\\ $\tau(\alpha_1)=\alpha_4$, $\tau(\alpha_2)=\alpha_5$, $\tau(\alpha_4)=\alpha_2$;
\item $\Gamma_1=\{\alpha_1, \alpha_2, \alpha_5\}$, $\Gamma_2=\{\alpha_4, \alpha_5, \alpha_1\}$,\\ $\tau(\alpha_1)=\alpha_4$, $\tau(\alpha_2)=\alpha_5$, $\tau(\alpha_5)=\alpha_1$;
\item $\Gamma_1=\{\alpha_1, \alpha_3, \alpha_4\}$, $\Gamma_2=\{\alpha_5, \alpha_2, \alpha_3\}$,\\ $\tau(\alpha_1)=\alpha_5$, $\tau(\alpha_3)=\alpha_2$, $\tau(\alpha_4)=\alpha_3$;
\item $\Gamma_1=\{\alpha_1, \alpha_4, \alpha_6\}$, $\Gamma_2=\{\alpha_5, \alpha_6, \alpha_2\}$,\\ $\tau(\alpha_1)=\alpha_5$, $\tau(\alpha_4)=\alpha_6$, $\tau(\alpha_6)=\alpha_2$;
\item $\Gamma_1=\{\alpha_1, \alpha_4, \alpha_6\}$, $\Gamma_2=\{\alpha_6, \alpha_2, \alpha_5\}$,\\ $\tau(\alpha_1)=\alpha_6$, $\tau(\alpha_4)=\alpha_2$, $\tau(\alpha_6)=\alpha_5$;
\item $\Gamma_1=\{\alpha_1, \alpha_2, \alpha_3, \alpha_4\}$, $\Gamma_2=\{\alpha_2, \alpha_3, \alpha_4, \alpha_5\}$,\\ $\tau(\alpha_1)=\alpha_2$, $\tau(\alpha_2)=\alpha_3$, $\tau(\alpha_3)=\alpha_4$, $\tau(\alpha_4)=\alpha_5$;
\item $\Gamma_1=\{\alpha_1, \alpha_2, \alpha_3, \alpha_6\}$, $\Gamma_2=\{\alpha_6, \alpha_3, \alpha_4, \alpha_5\}$,\\ $\tau(\alpha_1)=\alpha_6$, $\tau(\alpha_2)=\alpha_3$, $\tau(\alpha_3)=\alpha_4$, $\tau(\alpha_6)=\alpha_5$;
\item $\Gamma_1=\{\alpha_1, \alpha_2, \alpha_4, \alpha_6\}$, $\Gamma_2=\{\alpha_4, \alpha_5, \alpha_6, \alpha_2\}$,\\ $\tau(\alpha_1)=\alpha_4$, $\tau(\alpha_2)=\alpha_5$, $\tau(\alpha_4)=\alpha_6$, $\tau(\alpha_6)=\alpha_2$;
\item $\Gamma_1=\{\alpha_1, \alpha_2, \alpha_5, \alpha_6\}$, $\Gamma_2=\{\alpha_4, \alpha_5, \alpha_6, \alpha_1\}$,\\ $\tau(\alpha_1)=\alpha_4$, $\tau(\alpha_2)=\alpha_5$, $\tau(\alpha_5)=\alpha_6$, $\tau(\alpha_6)=\alpha_1$.
\end{itemize}


\medskip

\textbf{Acknowledgments.} The authors are grateful to Seidon Alsaody, Borys Kadets,
Yury Nikolaevsky, Grigori Rozenblioum, and Efim Zelmanov for helpful conversations and many valuable suggestions, and to Tatyana Chunikhina and Aleksey Chunikhin (Kharkiv National University) for help with computer programming.

\end{document}